\documentclass[twoside,12pt,reqno]{amsart}

\usepackage{amsmath,amsfonts,amsthm,amssymb,graphicx,color,tikz}
\usepackage{cite}
\usetikzlibrary{shapes,arrows}

\def\AAA{\mathcal{A}}
\def\BBB{\mathcal{B}}

\def\T{\mathbb{T}}
\def\R{\mathbb{R}}
\newcommand {\eps}  {\varepsilon}

\newcommand{\grad}{\nabla}

\newtheorem{thm}{Theorem}[section]

\newtheorem{prop}[thm]{Proposition}
\theoremstyle{definition}

\newtheorem{rem}{Remark}[section]

\numberwithin{equation}{section}
\numberwithin{thm}{section}
\numberwithin{rem}{section}
\numberwithin{defn}{section}
\newenvironment{acknowledgment}{{\flushleft \bf Acknowledgment:}}{}

\begin{document}
\footnotetext{{\it Date:} \today.
%
%
}

\title[AP scheme for kinetic models with singular limit]
{An asymptotic preserving scheme for kinetic models with singular limit}

\author{Alina Chertock}
\address{Alina Chertock\\ Department of Mathematics\\ North Carolina State University\\ Campus Box 8205, Raleigh NC 27695\\USA}
\email{chertock@math.ncsu.edu}

\author{Changhui Tan}
\address{Changhui Tan\\
Department of Mathematics\\
Rice University\\
6100 Main St., Houston TX 77005\\
USA}
\email{ctan@rice.edu}

\author{Bokai Yan}
\address{Bokai Yan\\ Department of Mathematics \\ University of California, Los Angeles \\ 520 Portola Plaza, Los Angeles, CA 90095 \\
USA}
\email{yanbokai@gmail.com}

\setcounter{page}{1}

\begin{abstract}
We propose a new class of asymptotic preserving schemes to solve kinetic equations with mono-kinetic singular limit. The main idea to deal
with the singularity is to transform the equations by appropriate
scalings in velocity. In particular, we study two biologically related kinetic 
systems. We derive the scaling factors, and prove that the rescaled solution does not have a singular limit, under appropriate spatial
non-oscillatory assumptions, which can be verified numerically by a newly developed asymptotic preserving scheme. We set up a few numerical
experiments to demonstrate the accuracy, stability, efficiency and asymptotic preserving property of the schemes.
\end{abstract}

\maketitle
\vspace*{-0.8cm}
\tableofcontents

\section{Introduction}
We consider the following type of kinetic equations,
\begin{equation}\label{eq:general}
\begin{split}
&\partial_tf_\eps+v\cdot\grad_x
f_\eps=\frac{1}{\eps}Q(f_\eps),\\
&f_\eps(0,x,v) = f^0(x,v).
\end{split}
\end{equation}
Here, $f_\eps=f_\eps(t,x,v)$ is the probability density function at
time $t\geq0$,  of space variable $x=(x_1,\ldots,x_d)^T\in\Omega$ and velocity 
$v=(v_1,\ldots,v_d)^T\in\R^d$, with spatial domain $\Omega=\R^d$ or $\T^d$. $Q$ is the interaction operator, which can be nonlinear in 
$f_\eps$ and nonlocal in $x$ and $v$.

The main property of the interaction operator $Q$ of our concern is that it has a \emph{mono-kinetic} equilibrium, namely
\[
Q(f)=0\quad\Leftrightarrow\quad
f(t,x,v)=\rho(t,x)\delta_{v=u(t,x)},
\]
where $\delta$ is the Dirac delta distribution, and $\rho(t,x)$ and $u(t,x)=(u_1(t,x),\ldots,u_d(t,x))^T$ are macroscopic density and 
velocity, respectively, satisfying
\[
\rho(t,x)=\int_{\R^d} f(t,x,v)dv,\quad \rho(t,x)u(t,x)=\int_{\R^d} v f(t,x,v)dv.
\]
Under this setup, one can formally let $\eps\to 0$ in \eqref{eq:general} and obtain an asymptotic solution, 
\begin{equation}\label{eq:asymp}
\lim_{\eps\to0}f_\eps(t,x,v)=\rho(t,x)\delta_{v=u(t,x)},
\end{equation}
which is the equilibrium of $Q$, and is singular in $v$.

In this paper, we shall focus on the following two interaction operators, both of which have interesting biological applications. The first 
model is called \emph{aggregation system}, where the interaction operator is defined as
\begin{equation}\label{eq:agg}
\begin{split}
Q(f)(t,x,v) &=\left[\int_\Omega\int_{\R^d}\grad_xK(x-y)f(t,y,v^*)dv^*dy\right]\cdot\grad_vf(t,x,v)\\
&+\grad_v\cdot(vf(t,x,v)).
\end{split}
\end{equation}
The operator consists two parts. The first term describes pairwise attraction-repulsion interactions, where $K$ is the interaction
potential. A natural biological assumption is that the strength of the interaction depends on the distance between two agents: attraction 
in large distance and repulsion in short distance. Hence, $K=K(r)$ is radial, and it is decreasing when $r$ is small and increasing when 
$r$ is large. The second term represents relaxation in velocity. This term is less biologically motivated, but plays a crucial role in
deriving an interesting asymptotic limit \cite{bodnar2005derivation}. In fact, the mono-kinetic asymptotic solution \eqref{eq:asymp} is
rigorously derived in \cite{jabin2000macroscopic}, see also \cite{fetecau2015first}. Furthermore, $(\rho,u)$ satisfy
\begin{equation}
\partial_t\rho+\grad_x\cdot(\rho u)=0,\quad u(t,x)=-\int_{\Omega}\grad_xK(x-y)\rho(t,y)dy.
\label{ageq}
\end{equation}
This limiting system is realized as the \emph{aggregation equation} which appears in various contexts related to biological aggregation
models. The equation has been intensively studied in the recent decade, and we refer to \cite{mogilner1999non,topaz2006nonlocal} and
references therein.

The second model is called \emph{3-zone system}, where
\begin{equation}\label{eq:3zone}
\begin{split}
Q(f)(t,x,v) = &\left[\int_\Omega\int_{\R^d}\grad_xK(x-y)f(t,y,v^*)dv^*dy\right]\cdot\grad_vf(t,x,v)\\
+&\grad_v\cdot\left(\int_\Omega\int_{\R^d}\phi(|x-y|)(v-v^*)f(y,v^*)f(x,v)dv^*dy\right).
\end{split}
\end{equation}
The artificial relaxation term in \eqref{eq:agg} is replaced by an \emph{alignment} term, which models pairwise interactions in the middle
range. The alignment force, proposed by Cucker and Smale in \cite{cucker2007emergent}, describes the so called \emph{flocking phenomenon}
that agents align their velocities to the neighbors. Here, $\phi(x)$ is the \emph{influence function} which represents the strength of
alignment between two agents. It naturally depends on the distance
between the agents, and decreases when the distance becomes larger.
We also assume that $\phi$ is bounded and Lipschitz. Without loss of
generality, we take
\[\|\phi\|_{L^\infty}=\phi(0)=1.\]
 The 
kinetic representation of Cucker-Smale model is derived in \cite{ha2008particle}, analyzed in
\cite{carrillo2010asymptotic,tan2017discontinuous}, and numerically
studied in \cite{rey2016exact,tan2017discontinuous}.
We refer readers to \cite{do2017global,poyato2017euler} for discussions on Cucker-Smale dynamics with
singular influence function.

The interaction operator \eqref{eq:3zone} combines long-range attraction, short-range repulsion and mid-range alignment. Such 3-zone
interaction framework is proposed in \cite{reynolds1987flocks}. It has been very successful in biological and ecological modeling, and it is
widely used in computer animations. As $\eps\to0$, the asymptotic limit of \eqref{eq:general} with interaction \eqref{eq:3zone} is
rigorously derived in \cite{fetecau2016first}, where mono-kinetic asymptotes \eqref{eq:asymp} is justified, with $(\rho,u)$ satisfying
\begin{equation}
\label{eq:limitsystem}
\begin{aligned}
&\partial_t\rho+\grad_x\cdot(\rho u)=0,\\
&\int_\Omega\phi(|x-y|)\rho(t,y)(u(t,x)-u(t,y))dy=-\int_\Omega\grad_xK(x-y)\rho(t,y)dy.
\end{aligned}
\end{equation}
The wellposedness theory of the limiting system \eqref{eq:limitsystem} is also established in \cite{fetecau2016first}, with the additional
equality on momentum conservation. The system serves as a more biologically relevant substitute to the aggregation equation \eqref{ageq}.

The goal of this paper is to design a universal numerical scheme for (\ref{eq:general}) that solves the equation in both the kinetic regime 
when $\eps=\mathcal{O}(1)$, and the fluid regime when $\eps\to0$. This type of numerical schemes is called \emph{asymptotic preserving (AP)} 
and was originally introduced in \cite{jin1999efficient}. The commutative diagram on the left hand side of Figure \ref{fig:AP} illustrates 
the AP property. A scheme $f_\eps^h$ that approximates the solution $f_\eps$ with discretization parameters $h$ is AP if its stability 
requirement on $h$ is independent of $\eps$, and if its limit $f^h$ when $\eps$ tends to zero consistently serves as an approximation of the 
limiting solution $f$. Therefore, the scheme can be automatically applied to the limiting equation simply by setting $\eps\to0$. 

AP schemes have been very successful in solving kinetic equations with different types of hydrodynamic limits, see, e.g.,
\cite{jin2010asymptotic} for a recent review of AP schemes. In the conventional kinetic equations and corresponding AP schemes, the limiting 
profile is usually given by a smooth Maxwellian distribution. Hence one can use fixed grid points in velocity discretization with a cutoff.
The study of kinetic equations with non-Maxwellian equilibrium has received attentions recently. AP schemes have been designed for the
kinetic equations with heavy-tail equilibrium
\cite{crouseilles2016numerical,crouseilles2016numerical_II,wang2016asymptotic,wang2017asymptotic}.

The equilibrium of the alignment operator $Q$ for our system \eqref{eq:general}, on the contrary,  is given by a $\delta$-distribution in 
velocity space. As $\eps$ becomes small, the solution $f_\eps$ becomes more and more singular. This addresses a major challenge in designing 
AP schemes for \eqref{eq:general} as its direct discretization can not achieve high accuracy and stability for small $\eps$ due to the fact 
that the limit solution is singular.

To overcome the difficulty, we apply a family of transformations $\mathcal{T}_\eps$ to the original system \eqref{eq:general}. As illustrated in Figure 
\ref{fig:AP}, $f_\eps$ is mapped to a new function
$g_\eps=\mathcal{T}_\eps[f_\eps]$. The aim is to find appropriate transformations so that the limiting solution 
$g=\lim_{\eps\to0}g_\eps$ is not singular, and thus an AP scheme for $g_\eps$ can be designed without worrying about the singularity.

\tikzstyle{block} = [rectangle, draw, fill=blue!5, text width=2em, text centered, rounded corners, minimum height=2em]
\tikzstyle{line} = [draw, -latex']
\begin{figure}[ht!]
\begin{tikzpicture}[node distance=6em]

\node [block, text width=12em, minimum height=9em] (fgraph) {};

\node [left of=fgraph, xshift=2.3em, yshift=2.5em] (fepsdel) {\large$f_\eps^h$};
\node [right of=fepsdel] (feps) {\large$f_\eps$};
\node [below of=fepsdel, node distance=5em](fdel) {\large$f^h$};
\node [right of=fdel]  (f) {\large$f$};

\path [line] (fepsdel) -- node[above]{\small$h\to0$}(feps);
\path [line] (fepsdel) -- node[right, yshift=.3em]{\small$\eps\to0$}(fdel);
\path [line] (feps) -- node[right, yshift=.3em]{\small$\eps\to0$}(f);
\path [line] (fdel) -- node[above]{\small$h\to0$}(f);

\node [block, text width=12em, minimum height=9em, right of=fgraph,
node distance=20em] (ggraph) {};

\node [left of=ggraph, xshift=2.3em, yshift=2.5em] (gepsdel) {\large$g_\eps^h$};
\node [right of=gepsdel] (geps) {\large$g_\eps$};
\node [below of=gepsdel, node distance=5em] (gdel) {\large$g^h$};
\node [right of=gdel] (g) {\large$g$};

\path [line] (gepsdel) -- node[above]{\small$h\to0$}(geps);
\path [line] (gepsdel) -- node[right, yshift=.3em]{\small$\eps\to0$}(gdel);
\path [line] (geps) -- node[right, yshift=.3em]{\small$\eps\to0$}(g);
\path [line] (gdel) -- node[above]{\small$h\to0$}(g);

\path [line] ([yshift=.5em]fgraph.east) -- node[above]{$\mathcal{T_\eps}$}([yshift=.5em]ggraph.west);
\path [line] ([yshift=-.5em]ggraph.west)-- node[below]{$\mathcal{T}^{-1}_\eps$}([yshift=-.5em]fgraph.east);

\end{tikzpicture}
\caption{AP scheme under transformation}\label{fig:AP}
\end{figure}
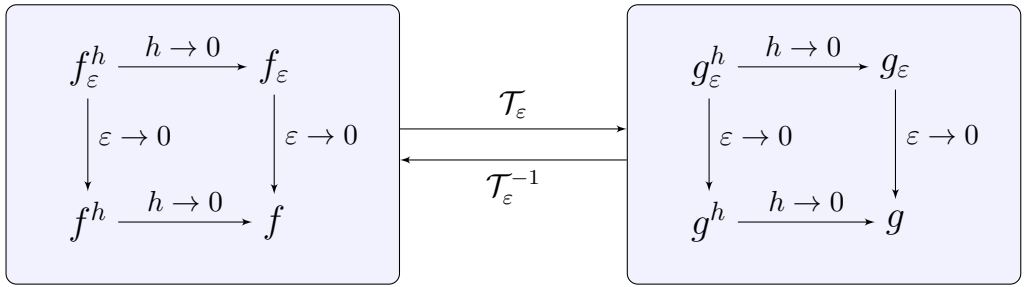

Since the singularity of $f$ is a $\delta$-distribution in velocity, a
natural choice of the transformation $\mathcal{T}_\eps$ is a scaling
in velocity. 
Velocity scaling methods have been used in various kinetic systems with singular time asymptotic limits, see, e.g., 
\cite{bobylev2000some,filbet2013rescaling,filbet2004rescaling,rey2016exact}. The heart of the matter in these methods is to find an 
appropriate \emph{scaling factor} $\omega_\eps$ to ensure that the rescaled function $g_\eps$ is not singular. In 
\cite{bobylev2000some,filbet2013rescaling,filbet2004rescaling}, the choice of $\omega_\eps$ is based on the self-similar behavior
of the spatial homogeneous equations, in which the transport part in \eqref{eq:general} is omitted. The scaling factor used in these works is 
proven to be optimal only for homogeneous systems with self-similar initial configurations. In \cite{rey2016exact}, a new scaling factor is
introduced for kinetic flocking systems and it is shown to be exact for spatially homogeneous systems with all smooth initial conditions.

In this paper, we present an AP scheme for \eqref{eq:general}
based on the velocity scaling method, where the transformation is
given by \eqref{eq:vs}, with a scaling factor similar 
to the one proposed in \cite{rey2016exact}. 
We study the asymptotic behavior of the rescaled function $g_\eps$, and provide sufficient conditions to ensure $g_\eps$ is non-singular uniformly in 
$\eps$. The result indicates that our choice of scaling factor captures the right scaling. Moreover, it implies that our numerical scheme is 
indeed asymptotic preserving.

The rest of the paper is organized as follows. In Section \ref{sec:vsm}, we describe the velocity scaling method, and show that with our 
choice of the scaling factor, the scaling is exact for the spatial ``homogeneous'' aggregation and 3-zone systems. In Section 
\ref{sec:asymp}, we discuss the asymptotic behavior of the full system \eqref{eq:general}, and prove that the rescaled profile is not 
singular under appropriate non-oscillatory conditions. In Section \ref{sec:AP}, we design AP schemes for the systems after velocity scaling, 
and discuss the AP property of \eqref{eq:general}. Finally, in Section \ref{sec:test}, we provide numerical experiments to illustrate the 
performance of the new scheme.

\section{Velocity scaling method}\label{sec:vsm}
In this section, we present the velocity scaling method for \eqref{eq:general}. We shall follow the storyline of \cite{rey2016exact} to 
derive a rescaled system.

\subsection{Exact rescaling on spatial ``homogeneous'' system}
As the main driving force of the system towards singularity is the interaction operator $Q$, we first consider the spatial ``homogeneous''
system
\begin{equation}\label{eq:flock}
\partial_tf_\eps=\frac{1}{\eps}Q(f_\eps),
\end{equation}
omitting the free transport part. We rescale the velocity variable by
\[\xi=\frac{v-u_\eps}{\omega_\eps},\]
and the transformation $\mathcal{T}_\eps$ is defined as
\begin{equation}\label{eq:vs}
g_\eps(t,x,\xi)=\mathcal{T}_\eps[f_\eps](t,x,\xi):=\omega_\eps^df_\eps(t,x,v)
=\omega_\eps^df_\eps(t,x,u_\eps+\omega_\eps\xi).
\end{equation}
Here, $\omega_\eps=\omega_\eps(t,x)$ is a \emph{scaling factor}, and $u_\eps=u_\eps(t,x)$ is the macroscopic velocity defined as
\begin{equation}\label{eq:u_eps}
u_\eps(t,x)=\frac{\int_{\R^d} vf_\eps(t,x,v)dv}{\int_{\R^d} f_\eps(t,x,v)dv}.
\end{equation}
The rescaled function $g_\eps$ has the following properties. First, the macroscopic density of $g_\eps$ is the same as the macroscopic
density of $f_\eps$,
\begin{equation}\label{eq:densitymatch}
\int_{\R^d} g_\eps(t,x,\xi)d\xi=\int_{\R^d} f_\eps(t,x,v)dv=:\rho_\eps(t,x).
\end{equation}
Second, with the shift by $u_\eps$, the first moment of $g_\eps$ in $\xi$ is zero for all $x$, namely the profile $g_\eps$ is nicely centered 
in $\xi$,
\begin{equation}\label{eq:zeromoment}
\int_{\R^d} g_\eps(t,x,\xi)\xi d\xi=0.
\end{equation}

An appropriate choice of scaling factor $\omega_\eps$ should produce a \emph{non-singular} rescaled function $g_\eps$, which neither 
concentrates nor spreads out as $\eps$ approaches zero, namely
\begin{equation}\label{eq:goodg}
\max_\xi|g_\eps(t,x,\xi)|\leq G,\quad\text{and}\quad\underset{\xi}{\text{supp}}g_\eps(t,x,\xi)\subset B_R(0),\quad
\forall x\in\Omega,\ \forall t\in[0,T],
\end{equation}
where $G$ and $R$ are finite and independent of $\eps$. Here, $B_R(0)$ denotes for a ball centered at origin and has radius $R$ in
$(\R^d,|\cdot|_\infty)$.

To choose an appropriate $\omega_\eps$, we represent the dynamics of $f_\eps$ by the triple $(g_\eps,u_\eps,\omega_\eps)$. 
The term $\partial_tf_\eps$ can then be expressed by
\begin{equation}\label{eq:dtf}
\begin{split}
\partial_tf_\eps=&-d\omega_\eps^{-d-1}\partial_t\omega_\eps g_\eps+\omega_\eps^{-d}\left[\partial_t g_\eps
-\grad_\xi g_\eps\cdot\frac{\omega_\eps\partial_tu_\eps+(v-u_\eps)\partial_t\omega_\eps}{\omega_\eps^2}\right]\\
 =&~\omega_\eps^{-d}\left[\partial_t g_\eps-\frac{\partial_t\omega_\eps}{\omega_\eps}\grad_\xi\cdot(\xi g_\eps)-\omega_\eps^{-1}\partial_tu_\eps\cdot\grad_\xi g_\eps\right],
\end{split}
\end{equation}
and the interaction kernel $Q$ is expressed as
\begin{equation}\label{eq:Qf}
Q(f_\eps)=\omega_\eps^{-d}\AAA_\eps\grad_\xi\cdot(\xi g_\eps)-\omega_\eps^{-d-1}\BBB_\eps\cdot\grad_\xi g_\eps.
\end{equation}
Here, $\AAA_\eps=\AAA_\eps(t,x)$ and $\BBB_\eps=\BBB_\eps(t,x)$ differ for different models. For aggregation system \eqref{ageq},
\begin{equation}\label{eq:aggpara}
\AAA_\eps(t,x)\equiv1, \quad \BBB_\eps(t,x)=-u_\eps(t,x)-\int_\Omega\grad_xK(x-y)\rho_\eps(t,y)dy.
\end{equation}
For 3-zone system \eqref{eq:limitsystem},
\begin{equation}\label{eq:3zonepara}
\begin{split}
\AAA_\eps(t,x)&=\int_\Omega\phi(|x-y|)\rho_\eps(t,y)dy,\\
\BBB_\eps(t,x)&=\int_\Omega\phi(|x-y|)\left[u_\eps(t,y)-u_\eps(t,x)\right]\rho_\eps(t,y)dy\\
&-\int_\Omega\grad_xK(x-y)\rho_\eps(t,y)dy.
\end{split}
\end{equation}

Combining \eqref{eq:dtf} and \eqref{eq:Qf}, we obtain an evolution equation for $g_\eps$
\begin{equation}\label{eq:homog}
\partial_tg_\eps=\left(\frac{\partial_t\omega_\eps}{\omega_\eps}+\frac{1}{\eps}\AAA_\eps\right)
\grad_\xi\cdot(\xi g_\eps)+\frac{1}{\omega_\eps}\left(\partial_tu_\eps-\frac{1}{\eps}\BBB_\eps\right)\cdot\grad_\xi g_\eps.
\end{equation}

To describe the dynamics of $\rho_\eps$ and $u_\eps$, we take zeroth and first moments of $f_\eps$ in \eqref{eq:flock}:
\begin{align*}
&\frac{\partial}{\partial t}\rho_\eps(t,x)=\frac{\partial}{\partial t}\int_{\R^d} f_\eps(t,x,v)dv=0,\\
&\frac{\partial}{\partial t}(\rho_\eps(t,x)u(t,x))=\frac{\partial}{\partial t}\int_{\R^d} f_\eps(t,x,v)vdv=\frac{1}{\eps}\rho_\eps\BBB_\eps,
\end{align*}
which in turn implies that 
\begin{equation}
\rho_\eps(t,x)=\rho^0(x)\quad\mbox{and}\quad\partial_tu_\eps-\frac{1}{\eps}\BBB_\eps=0.
\label{mom}
\end{equation}
Therefore, taking into account \eqref{mom} and defining the scaling factor $\omega_\eps$ in \eqref{eq:homog} as
\begin{equation}\label{eq:omegahomo}
\begin{cases}
\partial_t\omega_\eps=-\frac{1}{\eps}\omega_\eps\AAA_\eps\\
~\omega_\eps^0\equiv1\end{cases}
\quad\Rightarrow\quad
\omega_\eps(t,x)=\exp\left(-\frac{1}{\eps}\int\limits_0^t\AAA_\eps(s,x)ds\right),
\end{equation}
leads to $\partial_tg_\eps=0$, and thus $g_\eps$ remains unchanged in all time. In this case, we say that the rescaling is \emph{exact}
with factor $\omega_\eps$.

Since the initial profile $f^0(x)$ does not depend on $\eps$, it is easy to check that the initial triple $(g^0,u^0,\omega^0)$ is also
independent to $\eps$. Hence, the solution $g(t,x,\xi)=g^0(x,\xi)$ remains the same while $\eps$ varies and condition \eqref{eq:goodg} is
clearly satisfied as long as $g^0$ satisfies \eqref{eq:goodg}. Moreover, as the scaling is exact, we can easily reconstruct $f_\eps$ as
follows:
\[
f_\eps(t,x,v)=e^{dt\AAA_\eps(t,x)/\eps}f^0\left(x,e^{t\AAA_\eps(t,x)/\eps}(v-u_\eps(t,x))+u_\eps^0(x)\right).
\]
\begin{rem}
In the previous works \cite{filbet2004rescaling, filbet2013rescaling}, the scaling factor was chosen to be 
$\omega_\eps(t,x)=T_\eps(t,x)^{-1/2}$, where $T_\eps$ was the \emph{temperature} of the system, that is,
\[
T_\eps(t,x)=\frac{1}{d\rho_\eps(t,x)}\int|v-u_\eps(t,x)|^2f_\eps(t,x,v)dv.
\]
It was also shown that such scaling factor was exact for self-similar initial data. A new scaling factor proposed in \cite{rey2016exact}
takes advantage of the structure of the interaction operator and it is exact for all initial data.
\end{rem}

\subsection{Rescaling on the full system with free transport}
We now apply the scaling argument to the full system \eqref{eq:general}. The presence of free transport destroys the self-similar structure
of the spatial homogeneous system \eqref{eq:flock} and therefore it is in general impossible to find an exact scaling. We thus extend the
idea of the new scaling factor to find a non-singuar rescaled function, in the sense of \eqref{eq:goodg}.

The free transport term can be expressed in terms of $(g_\eps,u_\eps,\omega_\eps)$ as follows,
\[
\begin{aligned}
v\cdot\grad_xf_\eps&=(u_\eps+\omega_\eps\xi)\cdot\left[-d\omega_\eps^{-d-1}\grad_x\omega_\eps g_\eps\right.\\
&\left.+\omega_\eps^{-d}\left(\grad_xg_\eps-\frac{\grad_x\omega_\eps}{\omega_\eps}(\xi\cdot\grad_\xi)g_\eps
-\frac{1}{\omega_\eps}\sum\limits_{i=1}^d\partial_{\xi_i}g_\eps\grad_x(u_\eps)_i\right)\right].
\end{aligned}
\]
Adding this new contribution to \eqref{eq:homog}, yields
\begin{equation}
\begin{aligned}
\partial_tg_\eps&+(u_\eps+\omega_\eps\xi)\cdot\grad_xg_\eps\\
&=\left(\frac{\partial_t\omega_\eps}{\omega_\eps}+(u_\eps+\omega_\eps\xi)\cdot\frac{\grad_x\omega_\eps}{\omega_\eps}
+\frac{1}{\eps}\AAA_\eps\right)\grad_\xi\cdot(\xi g_\eps)\\
&+\frac{1}{\omega_\eps}\left(\partial_tu_\eps+(u_\eps+\omega_\eps\xi)\cdot\grad_xu_\eps
-\frac{1}{\eps}\BBB_\eps\right)\cdot\grad_\xi g_\eps.
\end{aligned}
\label{ge}
\end{equation}
Obtaining an exact scaling in this case would require finding a scaling factor $\omega_\eps$ that satisfies
\[
\frac{\partial_t\omega_\eps}{\omega_\eps}+(u_\eps+\omega_\eps\xi)\cdot\frac{\grad_x\omega_\eps}{\omega_\eps}+\frac{1}{\eps}\AAA_\eps=0.
\]
Since $\omega_\eps=\omega_\eps(t,x)$ is independent on the velocity variable $\xi$, such $\omega_\eps$ does not exist. Instead, we take 
$\omega_\eps$ which satisfies
\begin{equation}\label{eq:omega}
\partial_t\omega_\eps+u_\eps\cdot\grad_x\omega_\eps+\frac{1}{\eps}\omega_\eps\AAA_\eps=0.
\end{equation}
We again set $\omega^0(x)\equiv 1$, namely we do not perform scaling
at $t=0$.

By taking the first two moments moments of \eqref{eq:general}, we deduce the dynamics of macroscopic density $\rho_\eps$ and velocity 
$u_\eps$:
\begin{align}
&\partial_t\rho_\eps+\grad_x\cdot(\rho_\eps u_\eps)=0\label{eq:rho}\\
&\partial_t(\rho_\eps u_\eps)+\grad_x\cdot(\rho_\eps u_\eps\otimes u_\eps)+\grad_x\cdot(\omega_\eps^2P_\eps)
=\frac{1}{\eps}\rho_\eps\BBB_\eps,\label{eq:u}
\end{align}
where $P_\eps$ is the pressure tensor defined as
\begin{equation}
P_\eps(t,x):=\int_{\R^d}\xi\otimes\xi~g_\eps(t,x,\xi)d\xi.
\label{peps}
\end{equation}
Note that equation \eqref{eq:u} can be rewritten in the following non-conservative form:
\begin{equation}\label{eq:unon}
\partial_tu_\eps+u_\eps\cdot\grad_xu_\eps+\frac{1}{\rho_\eps}\grad_x\cdot
(\omega_\eps^2 P_\eps)=\frac{1}{\eps}\BBB_\eps,
\end{equation}
and the two forms are equivalent in the non-vacuum region where $\rho_\eps(x)>0$.

Taking into account  \eqref{eq:omega}, \eqref{eq:rho} and \eqref{eq:u}, equation \eqref{ge} can be rewritten as 
\begin{equation}\label{eq:g}
\begin{aligned}
\partial_tg_\eps&+(u_\eps+\omega_\eps\xi)\cdot\grad_xg_\eps\\
&=\left(\xi\cdot\grad_x\omega_\eps\right)
\grad_\xi\cdot(\xi g_\eps)+\left((\xi\cdot\grad_x)u_\eps\right)\cdot\grad_\xi g_\eps
-\frac{1}{\rho_\eps\omega_\eps}\left(\grad_x\cdot(\omega_\eps^2P_\eps)\right)
\cdot\grad_\xi g_\eps,
\end{aligned}
\end{equation}
or in the following equivalent conservative form:
\[
\partial_tg_\eps+\grad_x\cdot\left((u_\eps+\omega_\eps\xi)g_\eps\right) =
\grad_\xi\cdot\left[\left((\xi\cdot\grad_x\omega_\eps)\xi+(\xi\cdot\grad_x)u_\eps-
\frac{1}{\rho_\eps\omega_\eps}\left(\grad_x\cdot(\omega_\eps^2P_\eps)\right)\right)g_\eps\right].
\]

Unlike the spatial ``homogeneous'' system \eqref{eq:flock}, the rescaled function $g_\eps$ does change in time now, and it varies with 
different $\eps$. To validate our choice of scaling factor for the full system, it is important to check that $g_\eps$ satisfies \eqref{eq:goodg} uniformly in $\eps$, particularly when $\eps$ approaches zero.

\section{Asymptotic behavior}\label{sec:asymp}
This section is devoted to studying the asymptotic behavior of equations \eqref{eq:omega}--\eqref{eq:g}, as $\eps\to0$. The goal is to 
understand whether $g_\eps$ is non-singular under the proposed rescaling when $\eps$ is small. The result also supports the AP property of
the numerical scheme that will be discussed in Section \ref{sec:AP} below.

We denote 
\begin{equation}
G_\eps(t):=\max_{x,\xi}|g_\eps(t,x,\xi)|,
\label{ge1}
\end{equation}
and $R_\eps(t)$ be the smallest number such that
\begin{equation}
{\text{supp}}_\xi g_\eps(t,x,\xi)\subset B_{R_\eps(t)}(0).
\label{re}
\end{equation}
We also recall that $g_\eps$ is non-singular if condition \eqref{eq:goodg} is satisfied and hence we shall show that $G_\eps(t)$ and 
$R_\eps(t)$ are bounded independent of $\eps$, for all $t\in[0,T]$, under appropriate assumptions.


\subsection{Non-oscillatory assumptions}\label{sec:nonosc}
We start our discussion with two assumptions on the solution triple $(g_\eps, u_\eps, \omega_\eps)$. The first one is a spatially
non-oscillatory assumption on the rescaled function $g_\eps$,
\begin{equation}\label{eq:nonosc}
|\grad_xg_\eps(t,x,\xi)|\leq C_1 g_\eps(t,x,\xi),
\end{equation}
for all $t\in[0,T]$, $x\in\Omega$ and $\xi\in\R^d$, where the constant
$C_1$ is uniform in $\eps$. Condition \eqref{eq:nonosc} implies 
non-oscillatory bounds on macroscopic quantities. Indeed, for density $\rho_\eps$, we have
\begin{equation}\label{eq:nonoscrho}
|\grad_x\rho_\eps(t,x)|=\left|\int \grad_xg_\eps(t,x,\xi) d\xi\right|\leq C_1\rho_\eps(t,x).
\end{equation}
For pressure $P_\eps$, we have the following estimate
\begin{equation}\label{eq:pressureest}
|P_\eps(t,x)| = \left|\int \xi\otimes\xi g_\eps(t,x,\xi) d\xi\right|\leq
R_\eps^2(t)\rho_\eps(t,x),
\end{equation}
and condition \eqref{eq:nonosc} implies
\begin{equation}\label{eq:nonoscp}
|\grad_x P_\eps(t,x)| =\left|\int \xi\otimes\xi \grad_xg_\eps(t,x,\xi)
  d\xi\right|\leq C_1R_\eps^2(t)\rho_\eps(t,x).
\end{equation}
If condition \eqref{eq:nonosc} is violated, then $g_\eps$ becomes more oscillatory when $\eps$ gets smaller, in which case one can not 
expect to design AP numerical scheme for $g_\eps$. 

The second assumption is the Lipchitz apriori bound on the macroscopic velocity $u_\eps$,
\begin{equation}\label{eq:gradubound}
\|\grad_xu_\eps(t,\cdot)\|_{L^\infty(\Omega)}\leq C_2<\infty,
\end{equation}
for all $t\in[0,T]$, where the constant $C_2$ is uniform in $\eps$. 

It should be observed that taking $\eps\to 0$ in \eqref{eq:rho} and \eqref{eq:unon}, one can formally obtain the limiting system, for 
which condition \eqref{eq:gradubound} is also satisfied. The argument has been rigorously proved in \cite{jabin2000macroscopic} for the
aggregation system \eqref{ageq} and in \cite{fetecau2016first} for the 3-zone system \eqref{eq:limitsystem}. For both systems, the limiting
velocity $u$ is Lipschitz globally in time, under suitable regularity assumptions on kernels $K$ and $\phi$, and thus satisfies 
\eqref{eq:gradubound}. The regularity for the limiting system does not imply, however, that \eqref{eq:gradubound} holds uniformly in $\eps$. 
In fact, the convergence of $f_\eps$ to $\rho\delta_{v=u}$ is only weak-$\ast$ in measure. This does not rule out the possibility of 
oscillation in $x$ as $\eps\to0$.

In \cite{tadmor2014critical,carrillo2016critical}, it has also been proven that condition \eqref{eq:gradubound} is satisfied when the system 
\eqref{eq:rho}, \eqref{eq:unon} is considered in the pressureless regime and does not depend on $g_\eps$, i.e.,
\begin{equation}\label{eq:pressureless}
\partial_t\rho_\eps+\grad_x\cdot(\rho_\eps u_\eps)=0,\quad
\partial_tu_\eps+u_\eps\cdot\grad_xu_\eps=\frac{1}{\eps}\BBB_\eps,
\end{equation}
and subject to subcritical initial data. Moreover, the subcritical region becomes larger when $\eps$ gets smaller. Therefore,
\eqref{eq:gradubound} is satisfied uniformly for $\eps\in[0,\eps_0]$ if the initial profile $u^0$ lies in the subcritical region of the 
system \eqref{eq:pressureless} with $\eps=\eps_0$. \smallskip

The result for pressureless system \eqref{eq:pressureless} can be easily extended to the general dynamics \eqref{eq:unon} when the pressure
term is Lipschitz bounded uniformly in $\eps$. Indeed, from \eqref{eq:pressureest} and \eqref{eq:nonoscp}, we know $P_\eps/\rho_\eps$
and $\grad_x P_\eps/\rho_\eps$ are uniformly bounded by $R_\eps^2$. This together with the estimate on $\omega_\eps$ (see Section 
\ref{sec:omega}) implies boundedness of the pressure term in \eqref{eq:unon}. The Lipschitz bound can also be obtained by the additional 
non-oscillatory assumption \[|\grad_x^{\otimes2}g_\eps(t,x,\xi)|\lesssim g_\eps(t,x,\xi).\]
We omit the proof and redirect the reader to \cite{tadmor2014critical,carrillo2016critical} for relavant discussions.

We have thus argued that condition \eqref{eq:gradubound} holds under appropriate setup if $g_\eps$ is non-oscillatory in $x$. Therefore, 
both \eqref{eq:nonosc} and \eqref{eq:gradubound} are considered as spatially non-oscillatory assumptions on $g_\eps$. Given these two 
assumptions, we are going to prove that $g_\eps$ is non-singular, uniformly in $\eps$.

\subsection{The scaling factor} \label{sec:omega} In this section, we use the evolution equation \eqref{eq:omega} to estimate both the 
scaling factor $\omega_\eps$ and its gradient $\grad_x\omega_\eps$. To this end, we begin with the following proposition.
\begin{prop}\label{prop:omega}
Assume $\AAA_\eps$ is bounded below by a positive constant $c$ that is
independent of $\eps$, then
$\|\omega_\eps(t,\cdot)\|_{L^\infty(\Omega)}$ tends to 0 as $\eps\to0$.
\end{prop}
\begin{proof}
Consider a flow map $X_\eps(t,x)$ such that
\begin{equation}
\partial_tX_\eps(t,x)=u_\eps(t,X_\eps(t,x)),\quad X_\eps(0,x)=x.
\label{flowmap}
\end{equation}
Along each characteristic path, we have
\[\frac{d}{dt}\omega_\eps(t,X_\eps(t,x))=-\frac{1}{\eps}(\omega_\eps\AAA_\eps)(t,X_\eps(t,x)),\]
which in turns yields
\[\omega_\eps(t,X_\eps(t,x))=\omega^0(x)\exp\left(-\frac{1}{\eps}\int\limits_0^t\AAA_\eps(s,X_\eps(s,x)))ds\right)
\leq\exp\left(-\frac{c}{\eps}t\right).\]
Collecting all paths, we obtain
\begin{equation}\label{eq:omegalimit}
\|\omega_\eps(t,\cdot)\|_{L^\infty(\Omega)}\leq\exp\left(-\frac{c}{\eps}t\right),
\end{equation}
which vanishes as $\eps\to0$.
\end{proof}
\begin{rem}
Since $\lim_{\eps\to0}f_\eps$ is singular, a correct rescaling has to have a factor $\omega_\eps$ vanishes as $\eps\to0$. This is true for 
our choice of $\omega_\eps$.
\end{rem}
\begin{rem}\label{rem:Alow}
Under appropriate settings, the lower bound assumption on $\AAA_\eps$ is valid for both the aggregation system \eqref{eq:aggpara} and 
3-zone system \eqref{eq:3zonepara}. Indeed, for the aggregation system, $\AAA_\eps\equiv1$, while for the 3-zone system, $\AAA_\eps$ can be
estimated by
\[
\AAA_\eps(t,x)=\int_{\Omega}\phi(|x-y|)\rho_\eps(t,y)dy\geq\phi_{\min}\|\rho_\eps(t,\cdot)\|_{L^1(\Omega)}=\phi_{\min},
\]
provided $\phi$ is lower bounded by $\phi_{\min}>0$.
Note that $\|\rho_\eps(t,\cdot)\|_{L^1(\Omega)}=1$ due to mass
conservation, and the fact that $f_\eps$ is a probability distribution.
 The assumption on $\phi$ can be further relaxed (see e.g. \cite{tadmor2014critical}). 
We omit the details.
\end{rem}

Next, we provide a bound on $\grad_x\omega_\eps$, which is only needed for the 3-zone system \eqref{eq:3zonepara}, as the quantity is identically zero in aggregation system \eqref{eq:aggpara}.
\begin{prop}\label{prop:gradomega}
For the 3-zone system \eqref{eq:3zonepara}, we have $\|\grad_x\AAA_\eps\|_{L^\infty(0,T;\Omega)}\leq
C_1$. Moreover,
$\|\grad_x\omega_\eps\|_{L^\infty(0,T,\Omega)}$ tends to 0 as $\eps\to0$.
\end{prop}
\begin{proof}
We start with the estimate on $\grad_x\AAA_\eps$ and obtain from \eqref{eq:3zonepara}:
\begin{align*}
|\grad_x\AAA_\eps(t,x)|=&\left|\int\phi(|x-y|)\grad_x\rho_\eps(t,y)dy\right|\\
\leq& C_1 \int\phi(|x-y|)\rho_\eps(t,y)dy\leq
C_1\|\phi\|_{L^\infty}\|\rho_\eps(t,\cdot)\|_{L^1}=C_1,
\end{align*}
where the first inequality is due to non-oscillatory condition
\eqref{eq:nonoscrho}. Here, we recall our assumption that $\phi$ is
bounded and $\|\phi\|_{L^\infty}=1$.

We now estimate $\grad_x\omega_\eps$ by applying operator $\grad_x$ to equation
\eqref{eq:omega}. Once again we consider the flow map \eqref{flowmap} and obtain the following equation along each characteristic path:
\[\frac{d}{dt}\grad_x\omega_\eps(t,X_\eps(t,x))=-\frac{1}{\eps}\grad_x(\omega_\eps\AAA_\eps)
-\sum\limits_{j=1}^d\partial_{x_j}\omega_\eps\grad_x(u_\eps)_j.\]
Denote $M_\eps(t,x):=|\grad_x\omega_\eps(t,x)|_\infty$, where $|\cdot|_\infty$
is the infinity norm in $\R^d$. Then,
\[\frac{d}{dt}M_\eps(t,X_\eps(t,x))\leq\left(-\frac{c}{\eps}+|\grad_xu_\eps(t,X_\eps(t,x))|_\infty\right)M_\eps
+\frac{C_1}{\eps}\omega_\eps(t,X_\eps(t,x)).\]
This implies
\begin{align*}
M_\eps(t,X_\eps(t,x))\leq&
M(0,x)\exp\left[\int\limits_0^t\left(-\frac{c}{\eps}+|\grad_xu_\eps(x,X_\eps(s,x)|_\infty\right)ds\right]\\
+&\frac{C_1}{\eps}\int_0^t\omega_\eps(s,X_\eps(s,x))\exp\left[\int\limits_s^t\left(-\frac{c}{\eps}+|\grad_xu_\eps(\tau,X_\eps(\tau,x)|_\infty\right)d\tau\right]
   ds.
\end{align*}
As $\omega^0(x)\equiv 1$, we obtain $M(0,x)\equiv 0$.

Given any $t\in[0,T]$, we combine all paths and use \eqref{eq:gradubound}
\eqref{eq:omegalimit}, to obtain
\begin{align*}
\|M_\eps(t,\cdot)\|_{L^\infty}\leq&~
\frac{C_1}{\eps}\int\limits_0^t\exp\left(-\frac{c}{\eps}s\right)
\exp\left[-\frac{c(t-s)}{\eps}+C_2(t-s)\right]ds\\=&~
\frac{C_1
  (e^{C_2t}-1)}{C_2\eps}\exp\left(-\frac{c}{\eps}t\right),
\end{align*}
which also vanishes as $\eps\to0$.
\end{proof}

\subsection{The rescaled function}
We now investigate regularity properties of function $g_\eps$ in the sense of \eqref{eq:goodg}.

\begin{prop}\label{prop:gbound}
Consider functions $G_\eps(t)$ and $R_\eps(t)$ defined in \eqref{ge1} and \eqref{re}, respectively. Suppose $R_\eps(t)$ is uniformly 
bounded in $\eps$, for $t\in[0,T]$. Then, $G_\eps(t)$ is also uniformly bounded in $\eps$ for $t\in[0,T]$.
\end{prop}
\begin{proof}
Consider a flow map $(X_\eps(t,x,\xi), \Xi_\eps(t,x,\xi))$ in $(x,\xi)$ plane, where
\begin{equation}
\begin{aligned}
&\partial_tX_\eps(t,x,\xi) = u_\eps(t,X_\eps)+\omega_\eps(t,X_\eps)\Xi_\eps,\\
&\partial_t\Xi_\eps(t,x,\xi) = (\Xi_\eps\cdot\grad_x\omega_\eps(t,X_\eps))\Xi_\eps +
 (\Xi_\eps\cdot\grad_x)u_\eps(t,X_\eps)\\
 &\hspace*{2cm}-\frac{1}{\rho_\eps(t,X_\eps)\omega_\eps(t,X_\eps)}\grad_x\cdot(\omega_\eps^2(t,X_\eps)P_\eps(t,x)),\\
&X_\eps(0,x,\xi)=x,\quad \Xi_\eps(0,x,\xi)=\xi.
\end{aligned}
\label{a123}
\end{equation}
From \eqref{eq:g}, along each characteristic path we have
\[\frac{d}{dt}g_\eps(t,X_\eps(t,x,\xi),\Xi_\eps(t,x,\xi))=d~(\Xi_\eps\cdot\grad_x\omega_\eps(t,X_\eps))g_\eps(t,X_\eps,\Xi_\eps)).\]
Therefore, if $(x,\xi)\not\in supp(g^0)$, then $g_\eps(t,X_\eps,\Xi_\eps)=0$. If
$(x,\xi)\in supp(g^0)$, then
\begin{align*}
g_\eps(t,X_\eps,\Xi_\eps)=&g^0(x,\xi)\exp
\left[d\int\limits_0^t\Xi_\eps(s,x,\xi)\cdot\grad_x\omega_\eps(t,X_\eps(s,x))ds\right]\\
\leq&G(0)\exp\left[d\int\limits_0^tR_\eps(s)\|\grad_x\omega_\eps(s,\cdot)\|_{L^\infty(\Omega)}ds\right].
\end{align*}
Taking the supreme on all $(X_\eps,\Xi_\eps)$, yields
\[G_\eps(t)\leq G(0)\exp\left[d\int\limits_0^tR_\eps(s)\|\grad_x\omega_\eps(s,\cdot)\|_{L^\infty(\Omega)}ds\right].\]
Note that $G(0)$ does not depend on $\eps$ and therefore, from Proposition \ref{prop:gradomega} and the assumption that $R_\eps(t)$ is bounded uniformly in $\eps$, it follows that $G_\eps(t)$ is also uniformly bounded in $\eps$.
\end{proof}
\begin{rem}
It follows from Proposition \ref{prop:gbound} that if the rescaled function $g_\eps$ does not spread out, it will not concentrate either. In 
particular, for the aggregation system \eqref{eq:aggpara} there is no concentration regardless of the size of the support of $g_\eps$ since
$\grad_x\omega_\eps\equiv0$ in this case.
\end{rem}
We are left to prove that $g_\eps$ does not spread out. The growth of the support of $g_\eps$ is equivalent to the spread of the 
characteristic paths in \eqref{a123}, whose dynamics implies the following estimate on $R_\eps(t)$:
\begin{align*}
\frac{d}{dt}R_\eps(t)&\leq\|\grad_x\omega_\eps(t,\cdot)\|_{L^\infty(\Omega)}R_\eps(t)^2
+\|\grad_xu_\eps(t,\cdot)\|_{L^\infty(\Omega)}R_\eps(t)\\
&+2\|\grad_x\omega_\eps(t,\cdot)\|_{L^\infty(\Omega)}
\max_{x\in\Omega}\frac{|P_\eps(t,x)|_\infty}{\rho_\eps(t,x)}+\|\omega_\eps(t,\cdot)\|_{L^\infty(\Omega)}
\max_{x\in\Omega}\frac{|\grad_xP_\eps(t,x)|_\infty}{\rho_\eps(t,x)}.
\end{align*}
From the non-oscillatory bounds \eqref{eq:pressureest}--\eqref{eq:gradubound}, we obtain
\[
\frac{d}{dt}R_\eps(t)\leq \big(3\|\grad_x\omega_\eps(t,\cdot)\|_{L^\infty(\Omega)}+C_1\|\omega_\eps(t,\cdot)\|_{L^\infty(\Omega)} \big)R_\eps(t)^2+C_2R_\eps(t).
\]
The estimate has the form
\begin{equation}\label{eq:Restimate}
\frac{d}{dt}R_\eps(t)\leq a_\eps(t)R_\eps(t)^2+C_2R_\eps(t),
\end{equation}
where $a_\eps$ can be determined by Propositions \ref{prop:omega} and \ref{prop:gradomega}. 

The last inequality \eqref{eq:Restimate} allows us to prove that under appropriate assumptions on $\eps$ and $R(0)$, the function $R_\eps(t)$ in \eqref{re} is bounded globally in time.
\begin{prop}\label{prop:Rlongtime}
There exist $\eps_0>0$ and $R^0>0$ such that for all $\eps\in(0,\eps_0]$, function $R_\eps(t)$ defined in \eqref{re} is bounded for any 
finite time $t$ for both the aggregation system \eqref{eq:aggpara} and 3-zone system \eqref{eq:3zonepara} provided $R(0)\leq R^0$.
\end{prop}
\begin{proof}
We first consider aggregation system \eqref{eq:aggpara}. From the estimate \eqref{eq:omegalimit} and the fact that $\grad_x\omega_\eps=0$, 
we have
\[a_\eps(t)=C_1\exp\left(-\frac{c}{\eps}t\right).\]
We denote $S_\eps(t):=R_\eps(t)\exp(-\frac{c}{\eps}t)$ and obtain from \eqref{eq:Restimate} the dynamics of $S_\eps$:
\[\frac{d}{dt}S_\eps(t)\leq C_1S_\eps(t)^2-\left(\frac{c}{\eps}-C_2\right)S_\eps(t).\]
We now take $\eps_0=\dfrac{c}{2C_2}$ and $R^0=\dfrac{C_2}{C_1}$ and observe that since $c-C_2\eps>0$, $S_\eps$ has an invariant region 
$\left[0,\dfrac{c-C_2\eps}{C_1\eps}\right]$ and the following inequality holds:
\[
S(0)=R(0)\leq R^0=\frac{C_2}{C_1}=\frac{c-C_2\eps_0}{C_1\eps_0}\leq\frac{c-C_2\eps}{C_1\eps}.
\]
Therefore, $S_\eps(t)\leq\dfrac{c-C_2\eps}{C_1\eps}$ for all $t\geq0$ and we conclude with the bound
\[R_\eps(t)=S_\eps(t)\exp\left(\frac{c}{\eps}t\right)\leq
\frac{c-C_2\eps}{C_1\eps}\exp\left(\frac{c}{\eps}t\right).\]

Next, we turn our attention to the 3-zone system \eqref{eq:3zonepara}. By propositions \ref{prop:omega} and \ref{prop:gradomega}, we have
\begin{equation}\label{eq:3zonea}
a_\eps(t)=C_1\exp\left(-\frac{c}{\eps}t\right)
\left(1+\frac{3(e^{C_2t}-1)}{C_2\eps}\right).
\end{equation}
The extra exponential growth $e^{C_2t}$ in \eqref{eq:3zonea} is due to the estimate of $\grad_x\omega_\eps$. It can be controlled by the 
exponential decay provided $C_2<\dfrac{c}{\eps}$. In fact, we have
\[a_\eps(t)\leq\tilde{C}_1\exp\left(-\frac{c-C_2\eps}{\eps}t\right),
\quad\text{where }\tilde{C}_1=\max\left\{\frac{3C_1}{C_2\eps},C_1\right\}.\]
Using the same argument as the aggregation system, we obtain the following bound:
\[R_\eps(t)\leq\frac{c-2C_2\eps}{\tilde{C}_1\eps}
\exp\left(\frac{c-C_2\eps}{\eps}t\right)\]
as long as $\eps<\eps_0=\dfrac{c}{4C_2}$ and $R(0)\leq R^0=\dfrac{C_2}{6C_1}\cdot\min\{c,12\}$.
\end{proof}

It should be pointed out that the two bounds obtained above are not uniform in $\eps$. Uniform bounds can only be achieved up to a finite 
time, provided that $a_\eps(t)$ is uniformly bounded.
\begin{prop}\label{prop:Runif}
There exists a time $T>0$, such that $R_\eps(t)$ is bounded in $t\in[0,T]$, uniformly in $\eps$.
\end{prop}
\begin{proof}
For the aggregation system \eqref{eq:aggpara}, $a_\eps(t)$ is uniformly bounded by $C_1$ and from \eqref{eq:Restimate} we have
\[\frac{d}{dt}R_\eps(t)\leq C_1R_\eps(t)^2+C_2R_\eps(t),\]
which is a Ricatti-type first order ODE. Therefore, there exists a
finite time $T=T(C_1,C_2,R(0))>0$, such that $R_\eps(t)$ remains finite in
$[0,T]$. Since $T$ does not depend on $\eps$, the bound is uniformly in $\eps$.

For the 3-zone system \eqref{eq:3zonepara}, we use the estimate \eqref{eq:3zonea} to obtain
\[a_\eps(t)\leq C_1+\frac{3C_1(e^{C_2t}-1)}{C_2}\cdot\frac{1}{\eps}
\exp\left(-\frac{c}{\eps}t\right).\]
The $\dfrac{1}{\eps}$ term can be controlled by the exponentially decay, namely,
\[\frac{1}{\eps}\exp\left(-\frac{c}{\eps}t\right)\leq\frac{1}{cet},\]
for all $\eps\in[0,\infty)$, which in turns implies
\[a_\eps(t)\leq C_1+\frac{3C_1}{C_2ce}\cdot\frac{e^{C_2t}-1}{t}.\]
The right hand side in the last inequality is an increasing function in $t$. Therefore, we conclude that $a_\eps(t)$ is bounded by 
$C_1+\frac{3C_1}{C_2ce}\cdot\frac{e^{C_2T}-1}{T}$, which does not depend on $\eps$, and thus according to \eqref{eq:Restimate}, $R_\eps(t)$
is uniformly bounded in $\eps$ for any finite time $t\in[0,T]$.
\end{proof}

Putting everything together, we prove that $g_\eps$ is non-singular. It provides a strong support that our choice of $\omega_\eps$ captures 
the right scaling.

\begin{thm}\label{thm:regular}
Let $(g_\eps,u_\eps,\omega_\eps)$ be the solution triple of the rescaled dynamics \eqref{eq:omega}, \eqref{eq:u} and \eqref{eq:g}. Assume
the solution satisfies the non-oscillatory conditions \eqref{eq:nonosc} and \eqref{eq:gradubound}. Then, there exits a time $T=T(g^0)>0$, 
such that $g_\eps(t)$ is non-singular uniformly in $\eps$ in the sense of \eqref{eq:goodg}, for all $t\in[0,T]$.
\end{thm}

\section{Asymptotic preserving schemes}
\label{sec:AP}
Now we design a asymptotic preserving scheme to solve \eqref{eq:general}. To avoid the singularity limit, we use velocity scaling method
and express the solution $f_\eps$ by the rescaled function $g_\eps$, together with the scaling factor $\omega_\eps$ and macroscopic velocity 
$u_\eps$. We have shown in Theorem \ref{thm:regular} that under our proposed rescaling, $g_\eps$ is non-singular uniformly in $\eps$.
Therefore, we proceed to design AP schemes for the rescaled system, where singularity is no longer an obstacle.

\subsection{AP schemes for the rescaled systems}
Let us recall the dynamics of the solution triple $(g_\eps,u_\eps,\omega_\eps)$ and rewrite equations \eqref{eq:omega}, \eqref{eq:u}
and \eqref{eq:g} in the numerical  friendly conservative representations,
\begin{equation}
\label{eq:rescaledsystem}
\left\{
\begin{aligned}
&\partial_tg_\eps+\grad_x\cdot\left((u_\eps+\omega_\eps\xi)g_\eps\right)\\
&\hspace*{0.8cm} =
\grad_\xi\cdot\left[\left((\xi\cdot\grad_x\omega_\eps)\xi+(\xi\cdot\grad_x)u_\eps-
\frac{1}{\rho_\eps\omega_\eps}\left(\grad_x\cdot(\omega_\eps^2P_\eps)\right)\right)g_\eps\right],\\
&\partial_t(\rho_\eps u_\eps)+\grad_x\cdot(\rho_\eps u_\eps\otimes
u_\eps)+\grad_x\cdot(\omega_\eps^2P_\eps)=\frac{1}{\eps}\rho_\eps\BBB_\eps,\\
&\partial_t\omega_\eps+u_\eps\cdot\grad_x\omega_\eps+\frac{1}{\eps}\omega_\eps\AAA_\eps=0.
\end{aligned}
\right.
\end{equation}
with $\rho_\eps(t,x)$ defined in \eqref{eq:densitymatch} and  satisfying the continuity equation \eqref{eq:rho}.

To obtain an AP scheme for system \eqref{eq:rescaledsystem}, we introduce an increasing sequence $0<t^0<t^1\cdots<t^n\cdots$ of times with 
uniform  time step $\Delta t=t^{n+1}-t^n$ and denote by $q^n$ the value of any unknown quantity $q$ at time $t^n$, i.e., 
$q^n(\cdot)\approx q(t^n,\cdot)$. The canonical first order in time explicit-implicit time discretization for \eqref{eq:rescaledsystem} reads:
\begin{equation}
\label{scm:firstAP}
\left\{
\begin{aligned}
&\frac{g_\eps^{n+1} - g_\eps^n}{\Delta t} + \grad_x\cdot\left((u_\eps^n+\omega_\eps^n\xi)g_\eps^n\right) \\
&\qquad =
\grad_\xi\cdot\left[\left((\xi\cdot\grad_x\omega_\eps^n)\xi + (\xi\cdot\grad_x)u_\eps^n -
\frac{1}{\rho_\eps^n\omega_\eps^n}\left(\grad_x\cdot((\omega_\eps^n)^2P_\eps^n)\right)\right) g_\eps^n\right],\\
&\frac{\rho_\eps^{n+1} u_\eps^{n+1} - \rho_\eps^n u_\eps^n}{\Delta t}  + \nabla_x \cdot(\rho_\eps^n u_\eps^n \otimes u_\eps^n) + \nabla_x \cdot ((\omega_\eps^n)^2 P_\eps^n) = \frac{\rho_\eps^{n+1}}{\eps} \BBB_\eps^{n+1},\\
&\frac{\omega_\eps^{n+1} - \omega_\eps^n}{\Delta t} +u_\eps^n\cdot\grad_x\omega_\eps^n+\frac{1}{\eps}\omega_\eps^{n+1}\AAA_\eps^{n+1}=0,
\end{aligned}
\right.
\end{equation}
where the non-stiff fluxes are treated explicitly and the stiff terms are treated implicitly. To evolve the solution in time, we first
compute $g_\eps^{n+1}$ from the first equation in \eqref{scm:firstAP}, which is fully explicit as $g_\eps$ is non-singular, and its dynamics 
does not explicitly depend on $\eps$. Then, $\rho_\eps^{n+1}$ is obtained from the integration of $g_\eps^{n+1}$ in $\xi$ coordinate.
Next, we use an implicit solver to compute $\rho_\eps^{n+1}u_\eps^{n+1}$ from the second equation. Noting that the operator 
$\rho_\eps\BBB_\eps$ is a symmetric operator on $\rho_\eps u_\eps$, one can simply apply a conjugate-gradient method.
Finally, $\omega_\eps^{n+1}$ can be obtained easily from the third equation since $\AAA_\eps^{n+1}$ only depends on $\rho_\eps^{n+1}$ and 
hence can be computed explicitly.

One can derive a second order time discretization scheme by applying, say, a backward differentiation formula (BDF) on the time derivative, 
an extrapolation on the explicit terms and a fully implicit solver on the stiff terms. We omit the details here and refer the reader to \cite{goudon2013asymptotic,goudon2014asymptotic,wang2016asymptotic}.

A fully discrete scheme should be obtained by consistent spatial and velocity discretizations, for instance, by using a finite volume method
thanks to the conservative structure of the equations; see, e.g., \cite{rey2016exact} for the references. Importantly, since $g_\eps$ is
non-singular, the discretizations are independent of $\eps$.


We summarize the entire procedure of the proposed numerical approach for solving \eqref{eq:general}. Given initial data $f^0$, we set 
$\omega^0\equiv1$, compute $u^0$ by \eqref{eq:u_eps} and $g^0$ by performing velocity scaling transformation $\mathcal{T}_\eps$ in 
\eqref{eq:vs}. Then, we evolve the dynamics \eqref{eq:rescaledsystem} on $(g_\eps,u_\eps,\omega_\eps)$ using appropriate AP scheme, for 
instance \eqref{scm:firstAP}, until a target time $t$. Finally, we apply the inverse transformation $\mathcal{T}_\eps^{-1}$ to obtain the 
solution $f_\eps$ at time $t$. Note that $\mathcal{T}_\eps^{-1}$ has an explicit form
\begin{equation}\label{eq:Tinv}
f_\eps(t,x,v)=\mathcal{T}_\eps^{-1}[g_\eps](t,x,v)=\frac{1}{\omega_\eps(t,x)^d}
g_\eps\left(t,x,\frac{v-u_\eps(t,x)}{\omega_\eps(t,x)}\right),
\end{equation}
which is easy to implement numerically.

\subsection{Asymptotic preserving property}
Now, we verify the AP property of our numerical scheme.

Recall the limiting system of \eqref{eq:general} as $\eps\to0$ satisfies mono-kinetic asymptotes \eqref{eq:asymp} 
$f(t,x,v)=\rho(t,x)\delta_{v=u(t,x)}$, with macroscopic quantities $(\rho,u)$ satisfying
\begin{equation}\label{eq:limiting}
\partial_t\rho+\grad_x\cdot(\rho u)=0,\quad\rho\BBB=0.
\end{equation}

The goal is to check that $f_\eps^n$ converges to $f^n$ as $\eps\to0$, at the discrete level.

As discussed in Section \ref{sec:asymp}, the spatial non-oscillatory conditions \eqref{eq:nonosc} and \eqref{eq:gradubound} play an 
important role and guarantee that $g_\eps$ is non-singular, in the
sense of \eqref{eq:goodg}. This argument, stated in Theorem \ref{thm:regular}, can
be extended to semi-discrete or fully discrete dynamics with appropriate choices of discretizations.

We shall consider the first order scheme \eqref{scm:firstAP} as an example. The semi-discrete version of Theorem \ref{thm:regular} implies
that $g_\eps^n$ is non-singular, namely
\begin{equation}\label{eq:discreteg}
\max_\xi|g_\eps^n(x,\xi)|\leq G,\quad\text{and}\quad\underset{\xi}{\text{supp}}g_\eps^n(x,\xi)\subset B_R(0),\quad
\forall x\in\Omega,
\end{equation}
where $G, R$ are constants independent of $\eps$, if
\begin{equation}\label{eq:discretecond}
|\grad_xg_\eps^n(x,\xi)|\leq C_1g_\eps^n(x,\xi),\quad
\|\grad_xu_\eps^n\|_{L^\infty(\Omega)}\leq C_2,
\end{equation}
for all $x\in\Omega, \xi\in\R^d$,
where $C_1, C_2$ are constants which do not depend on $\eps$.

Assuming \eqref{eq:discretecond} holds, we, first, check $f_\eps^n$ converges to a mono-kinetic profile as $\eps\to0$. It is enough to show 
that for any given $x\in\Omega$, the size of $\text{supp}_vf_\eps^n(x,v)$ tends to 0 as $\eps\to0$. From \eqref{eq:Tinv} and 
\eqref{eq:discreteg}, we obtain
\[|\underset{v}{\text{supp}}f_\eps^n(x,v)|=
\omega_\eps^n(x)|\underset{\xi}{\text{supp}} g_\eps^n(x,\xi)|\leq
2R\omega_\eps^n(x).\]
A semi-discrete version of proposition \ref{prop:omega} implies that $\omega_\eps^n(x)\to0$ as $\eps\to0$, which finishes the proof.
Indeed, 
\[
\|\omega^n_\eps\|_{L^\infty(\Omega)}\leq\frac{\|\omega^{n-1}_\eps\|_{L^\infty(\Omega)}}{1+\frac{\Delta t}{\eps}\mathcal{A}^n}\leq
\frac{\|\omega^{n-1}_\eps\|_{L^\infty(\Omega)}}{1+\frac{c\Delta t}{\eps}}\leq\cdots\leq
\frac{\|\omega^0\|_{L^\infty(\Omega)}}{\left(1+\frac{c\Delta t}{\eps}\right)^n}\leq\exp\left(-\frac{ct}{\eps}\right) \stackrel{\eps\to0}{\longrightarrow}0.
 \]
Here, $\AAA_\eps^n\geq c$ due to Remark \ref{rem:Alow}.

Next, we show that the macroscopic quantities $(\rho_\eps^n, u_\eps^n)$ converges to $(\rho^n, u^n)$, which solves the semi-discrete version
of the limiting system \eqref{eq:limiting}:
\begin{equation}\label{eq:semilimiting}
\frac{\rho^{n+1}-\rho^n}{\Delta t}+\grad_x\cdot(\rho^n u^n)=0,\quad\rho^n\BBB^n=0.
\end{equation}
To this end, we integrate the $g_\eps^n$ equation in \eqref{scm:firstAP} with respect to $\xi$ to obtain
\[\frac{\rho_\eps^{n+1}-\rho_\eps^n}{\Delta
  t}+\grad_x\cdot(\rho_\eps^n u_\eps^n)=0.\]
Clearly, the limiting system as $\eps\to0$ is the first equation in \eqref{eq:semilimiting}.

For the second equation in \eqref{eq:semilimiting}, we rewrite the $u_\eps^n$ equation in \eqref{scm:firstAP} as follows
\[\rho_\eps^{n+1} \BBB_\eps^{n+1}=\eps\left[\frac{\rho_\eps^{n+1}
    u_\eps^{n+1} - \rho_\eps^n u_\eps^n}{\Delta t}  + \nabla_x
  \cdot(\rho_\eps^n u_\eps^n \otimes u_\eps^n) + \nabla_x \cdot
  ((\omega_\eps^n)^2 P_\eps^n)\right],\]
where the right hand side is of order $\mathcal{O}(\eps)$, thanks to the non-oscillatory condition \eqref{eq:discretecond}. Taking the limit
$\eps\to0$, we obtain $\rho^{n+1} \BBB^{n+1}=0$.

It should be observed, that the AP property can be also verified for full discrete schemes. Detailed discretization can be found in, e.g. 
\cite{rey2016exact}.

Note that the discrete non-oscillatory conditions \eqref{eq:discretecond} can be monitored during numerical simulations. Practically, 
instead of monitoring the oscillation on $g_\eps^n(x,\xi)$ for all $x\in\Omega, \xi\in\R^d$, we only need to keep track of the oscillation 
for $\rho$ and $P$, namely the discrete version of conditions \eqref{eq:nonoscrho} and \eqref{eq:nonoscp}. The AP property is guaranteed to 
hold as long as there is no violation of the following assumptions
\begin{equation}\label{eq:nonosce}
\max_x\frac{|\grad_x\rho_\eps^n|}{\rho_\eps^n},~~
\max_x\frac{|\grad_xP_\eps^n|}{\rho_\eps^n},~~
\max_x\grad_xu_\eps^n~\leq C,
\end{equation}
where $C$ is a constant which does not depend on $\eps$.

\section{Numerical experiments}\label{sec:test}
In this section, we demonstrate the performance of the proposed schemes on a number of numerical examples. We note that the
velocity scaling method in Section \ref{sec:vsm} and the resulting AP scheme in Section \ref{sec:AP} are dimension independent. For 
simplicity, numerical simulations are performed on a 1-D by 1-D phase
space, with periodic spatial domain $\Omega=\T=[-\pi,\pi]$. 
In particular, we consider the computation domain 
$(x,\xi)\in[-\pi,\pi]\times [-6,6]$, and pick initial data such that 
$R$ in \eqref{eq:goodg} is much smaller than 6. So, the solution will
vanish at the boundary.
Unless otherwise specified, we always take 
$N_x = 128$ and $N_\xi = 64$ grid points in the phase space. We take $\Delta t = \Delta x / 20$ to satisfy the CFL condition, where 
$\Delta x$ is spatial mesh size.

In this section, we focus on the 3-zone system (\ref{eq:3zone}). The aggregation system (\ref{eq:agg}) can be solved similarly. The alignment kernel $\phi$  of the 3-zone system is given by
\[\phi(x) = \frac{1}{\sqrt{1+x^2}}.\]
In the Examples 1--3 below, the interaction is modeled by the Morse potential
\[K(x) = -e^{-|x|/2} + e^{-|x|}.\]

\subsection{Example 1 -- Validation of the assumptions}
The first test is to check whether the spatial non-oscillatory assumption \eqref{eq:nonosce} is valid for a typical initial value problem of
\eqref{eq:general}.
The rescaled system \eqref{eq:rescaledsystem} is numerically solved subject to the initial data
\[
g^0(x,\xi) = \rho^0(x) M(\xi), \quad \rho^0(x) = 1+e^{-20(x-1)^2}+e^{-20(x+1)^2}, \quad u^0(x) = 0, \quad \omega^0(x) = 1,
\]
where $M(\xi) = \frac{1}{\sqrt{2\pi}} e^{-\xi^2/2}$.

We track the time evolution of $\max\limits_x\frac{|\grad_x \rho_\eps|}{\rho_\eps}$,
$\max\limits_x\frac{|\grad_x P_\eps|}{\rho_\eps}$ and $\max\limits_x|\grad_x u_\eps|$, for different values of $\eps$. The results shown in Figure \ref{fig:checkC1C2} suggest that the assumption \eqref{eq:nonosce} is valid and the bounds are uniform with respect to $\eps$. 
\begin{figure}[ht!]
\centerline{\includegraphics[width=.325\textwidth]{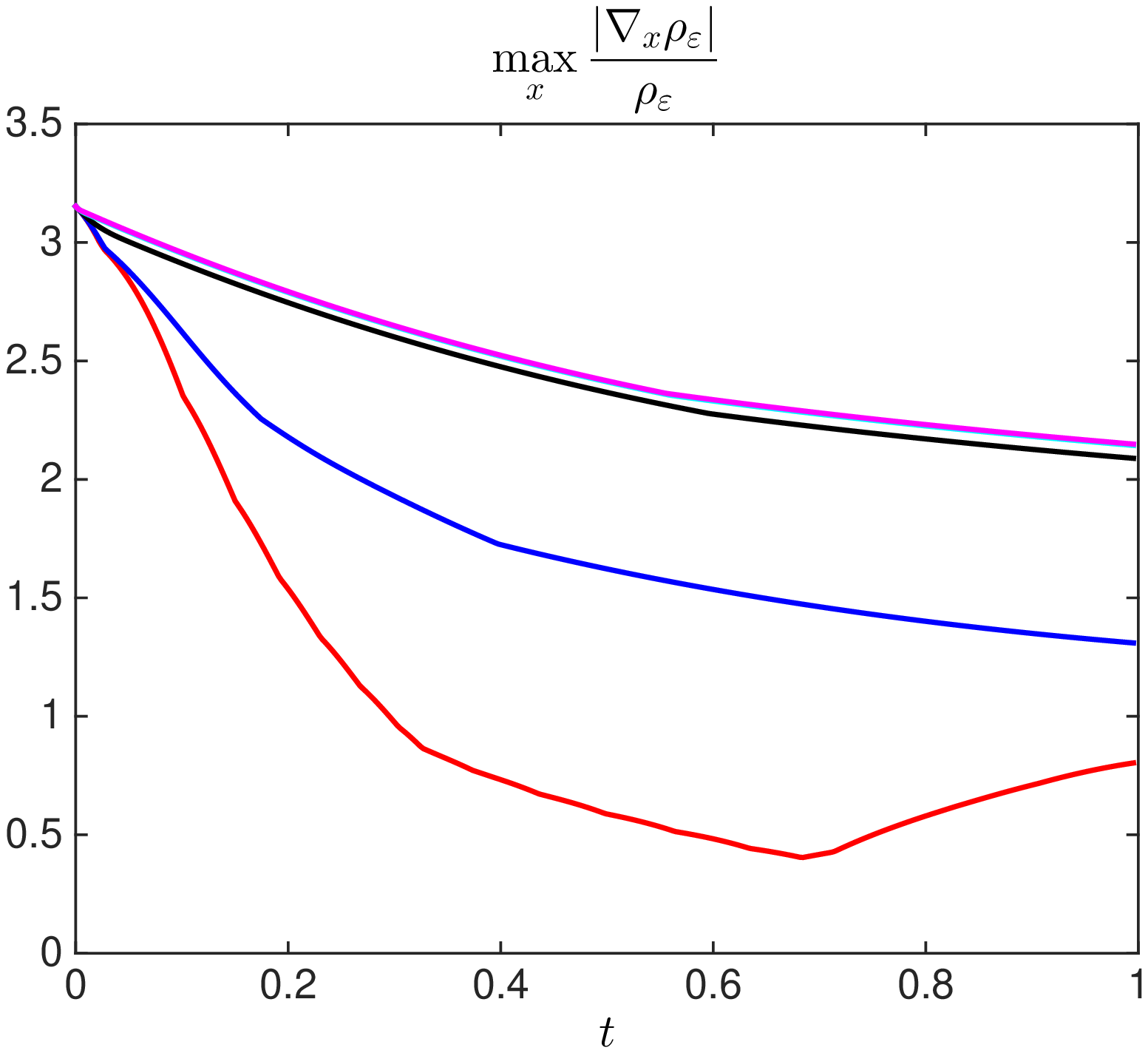}\includegraphics[width=.325\textwidth]{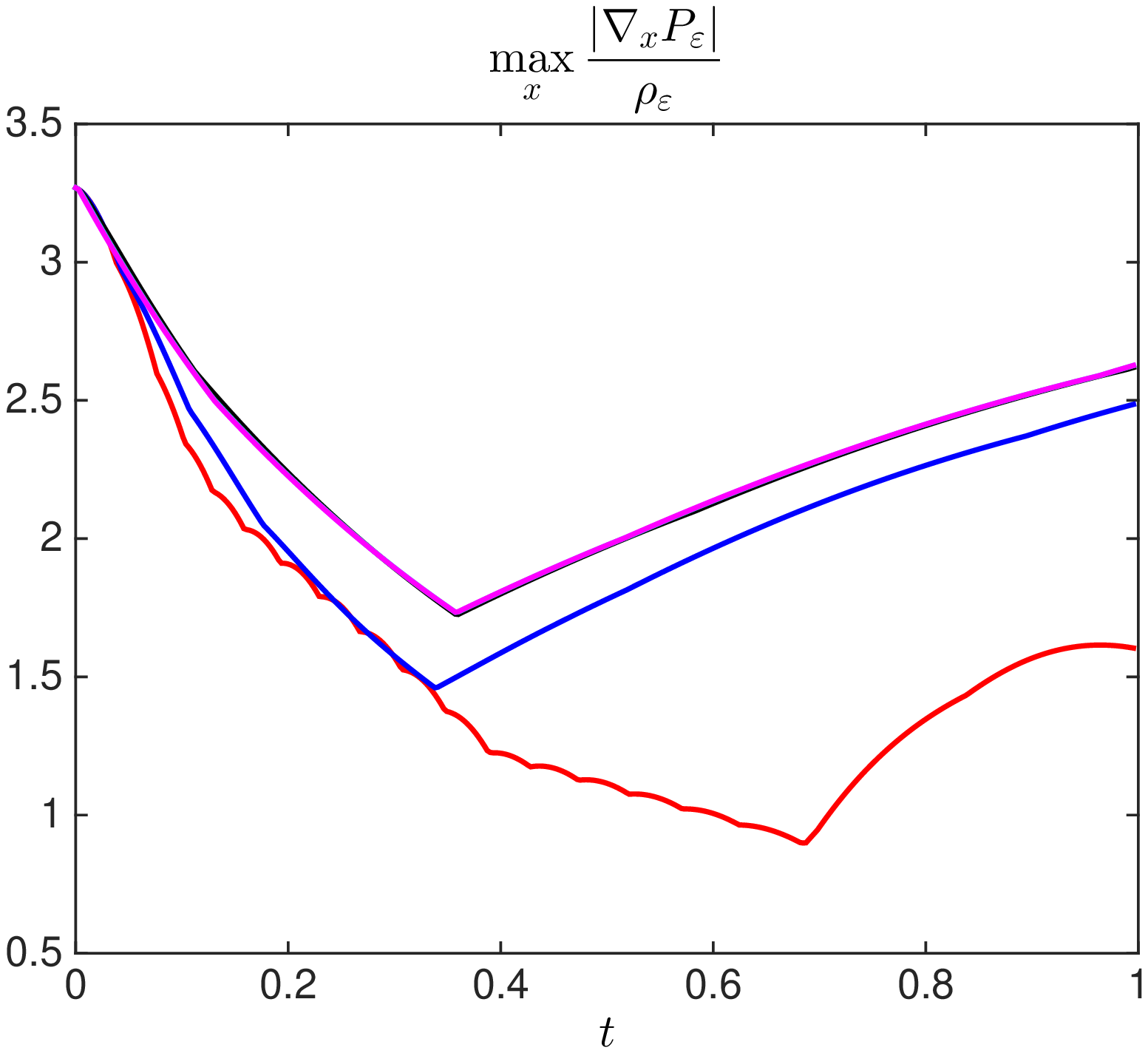}
\includegraphics[width=.325\textwidth]{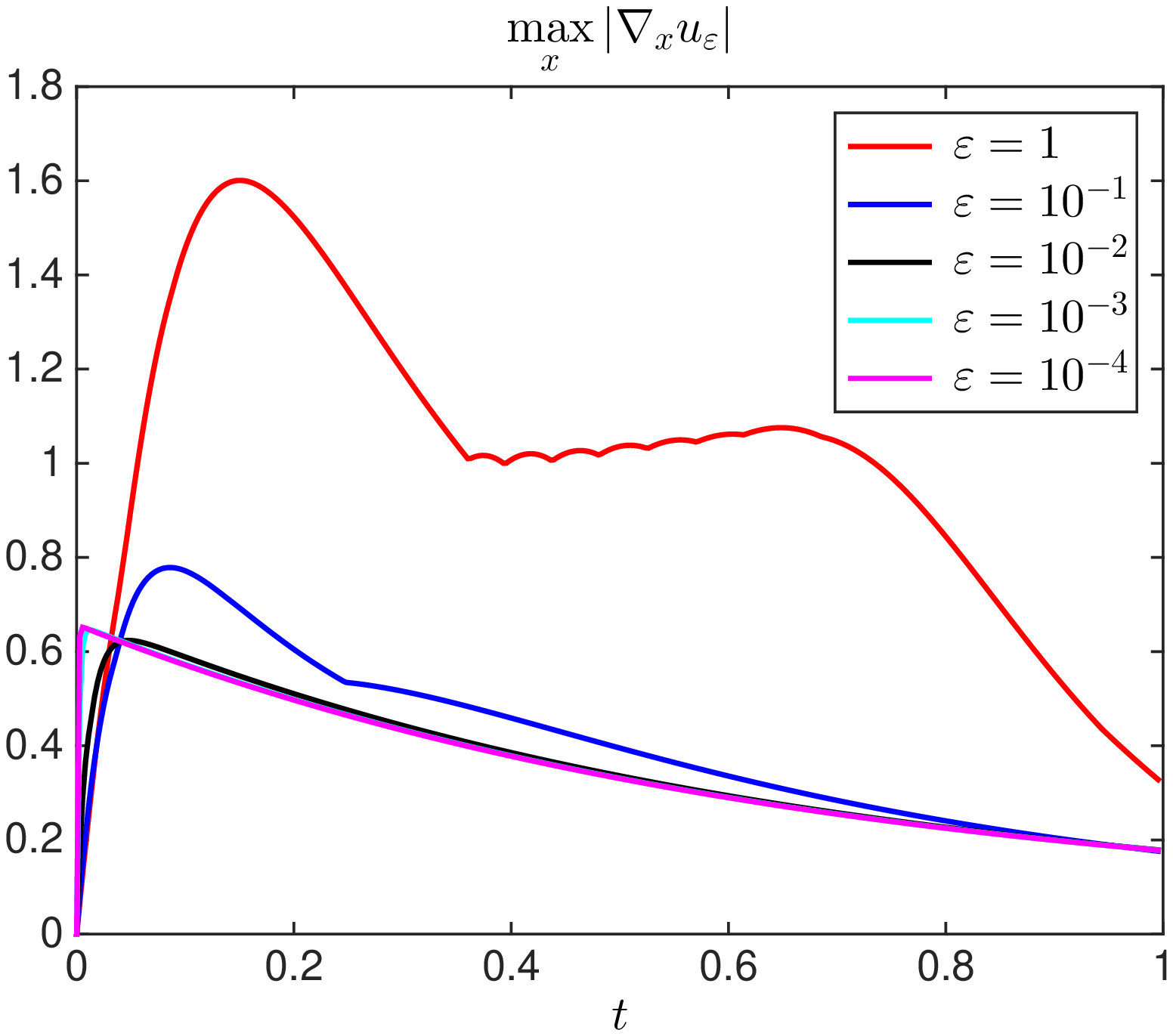}}
\caption{Example 1: The test on assumption \ref{eq:nonosce}. From left to right: the 
time evolution of $\max\limits_x |\grad_x u_\eps|$, $\max\limits_x \frac{|\grad_x \rho_\eps|}{\rho_\eps}$ and 
$\max\limits_x \frac{|\grad_x P_\eps|}{\rho_\eps}$ for different values of $\eps$. The lines for $\eps=10^{-3}$ and $\eps=10^{-4}$ are almost overlapped.}
\label{fig:checkC1C2}
\end{figure}

\subsection{Example 2 -- Consistency test}
In this example, we verify that the solution to the rescaled system \eqref{eq:rescaledsystem} is consistent with the original system 
\eqref{eq:general}. The original system is integratedsimulations in time by the forward Euler method, while the rescaled system is evolved by the AP
scheme  (\ref{scm:firstAP}). The original system is very difficult to solve for small $\eps$ and long time, due to the fact that the solution $f_\eps$ is approaching a singular delta function in velocity space. Hence, we taksimulationse $\eps=1$ and run the simulations until
the final time $t=0.7$ in this test. $N_v = 512$ points are used in solving the original system (compare with $N_\xi = 64$ for the rescaled system).

The following initial condition for the original system \eqref{eq:general} is used,
\begin{equation*}
\label{eq:numeric_ini_consistency}
f^0(x,v)= \frac{\rho^0(x)}{2\sqrt{0.4\pi}} \left( e^{-\frac{(v + \sin(x))^2}{0.4}} + e^{-\frac{(v - \sin(x))^2}{0.4}} \right),\quad 
\rho^0(x)= 1+e^{-20(x-1)^2}+\frac{3}{2}e^{-20(x+1)^2},
\end{equation*}
which is equivalent to the rescaled system \eqref{eq:rescaledsystem} solved subject to the initial condition
\begin{equation*}
\label{eq:numeric_ini_g_consistency}
\begin{aligned}
&g^0(x,\xi)= \frac{\rho^0(x)}{2\sqrt{0.4\pi}} \left( e^{-\frac{(\xi + \sin(x))^2}{0.4}} + e^{-\frac{(\xi-\sin(x))^2}{0.4}}\right),\\
&\rho^0(x)= 1+e^{-20(x-1)^2}+\frac{3}{2}e^{-20(x+1)^2},\quad u^0(x)=0, \quad \omega^0(x) = 1.
\end{aligned}
\end{equation*}

Time snapshots of the density $\rho_1(x)$ and the macroscopic velocity $u_1(x)$ at different time $t$ are compared in the top of Figure \ref{fig:consistency}. The solutions to different systems are almost identical, demonstrating that the rescaled system is consistent with the original system. In the bottom of Figure \ref{fig:consistency}, we show the distributions $f_1(x,v)$ (left, solved from the original system) and $g_1(x,\xi)$ (right, solved from the rescaled system) at time $t=0.7$. As one can see, $f_1(x,v)$ is getting concentrated in the velocity space, making it difficult to simulate with fixed grid points. In contrast, $g_1(x,\xi)$ has a finite support in the rescaled velocity space $\xi$ and a fixed grid in $\xi$ can be used for the simulation.
\begin{figure}[ht!]
    \includegraphics[width=.48\textwidth]{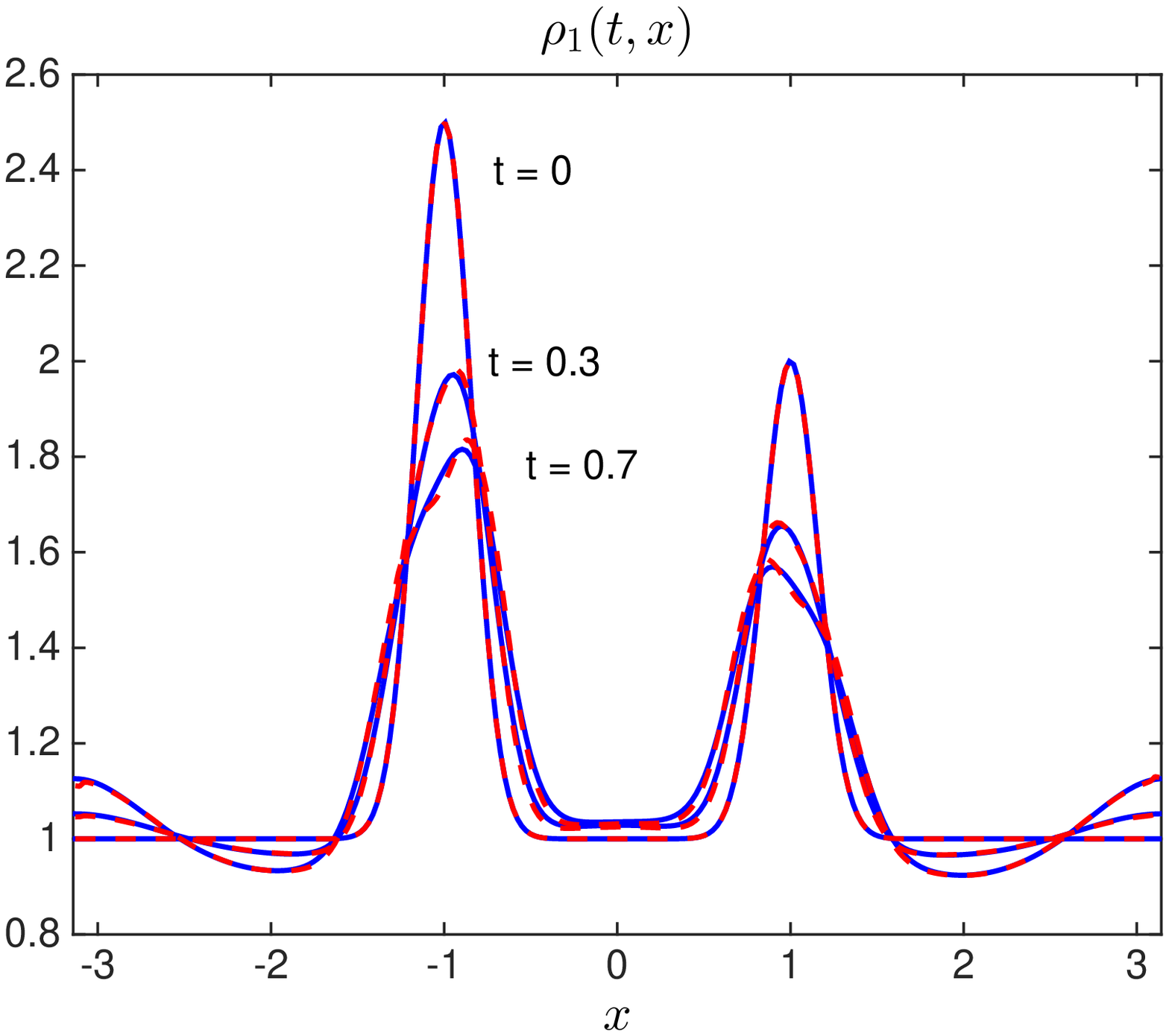}
    \includegraphics[width=.48\textwidth]{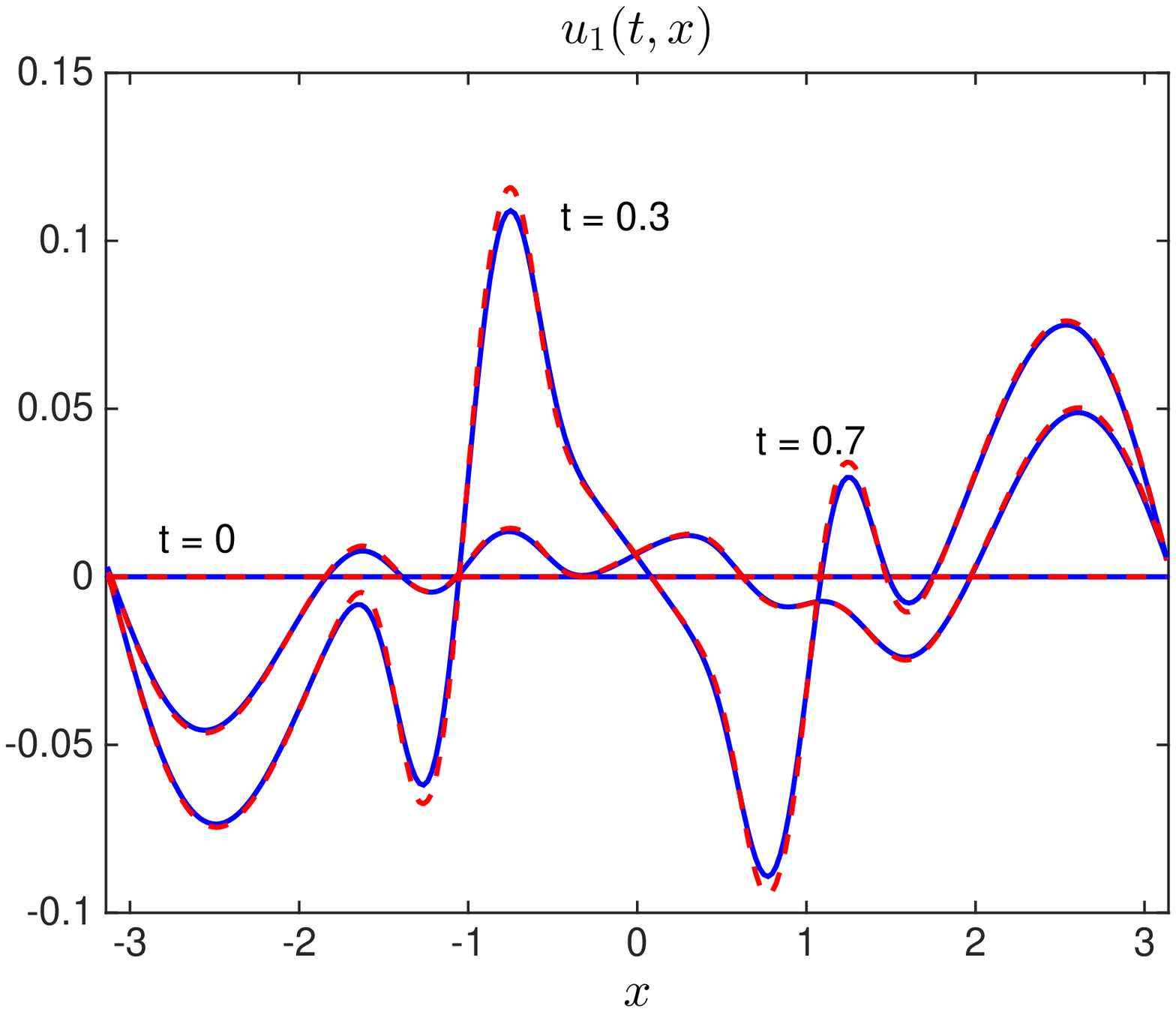}
    \includegraphics[width=.48\textwidth]{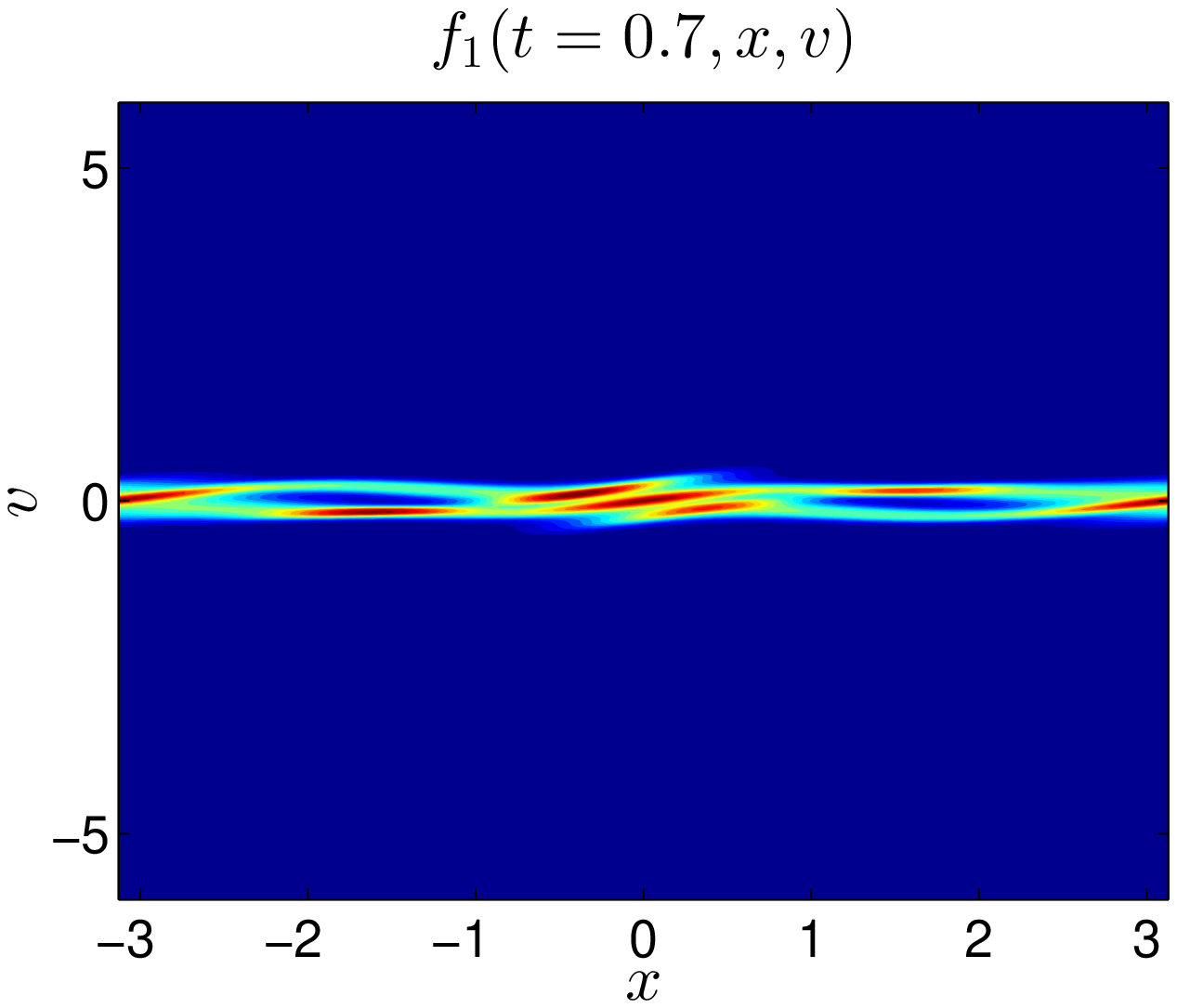}
    \includegraphics[width=.48\textwidth]{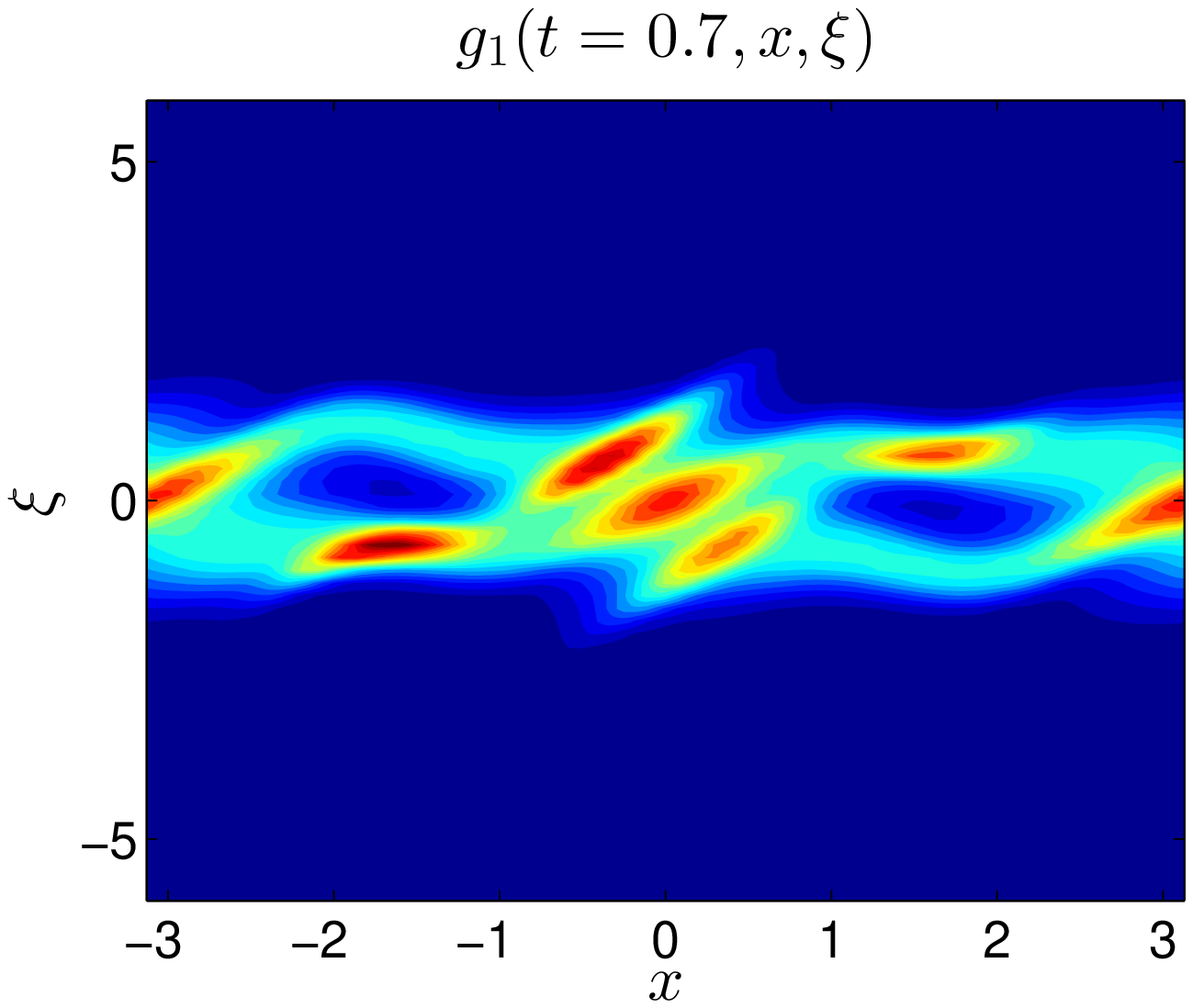}
\caption{Example 2: Top left: the time evolution of $\rho_1(x)$ solved from the original system \eqref{eq:general} (blue solid lines) and 
the rescaled system \eqref{eq:rescaledsystem} (red dashed lines). Top right: the time evolution of $u_1(x)$ solved from the original system 
\eqref{eq:general} (blue solid lines) and the rescaled system \eqref{eq:rescaledsystem} (red dashed lines). Bottom left: the distribution 
$f_1(x,v)$ at time $t=0.7$ solved from the original system \eqref{eq:general}. Bottom right: the distribution $g_1(x,\xi)$ at time $t=0.7$ 
solved from the rescaled system \eqref{eq:rescaledsystem}.}
\label{fig:consistency}
\end{figure}

\subsection{Example 3 - Asymptotic preserving test}
Now we test the AP property of the scheme (\ref{scm:firstAP}). More specifically, we compare the solutions of (\ref{scm:firstAP}) with vanishing $\eps$ to the solution of the limiting system (\ref{eq:limitsystem}). We use the following initial data
\[g^0(x,\xi) = \rho^0(x) M(\xi), \quad \rho^0(x) = 0.01+e^{-20 x^2}, \quad u^0(x) = 0, \quad \omega^0(x) = 1,\]
for the scheme (\ref{scm:firstAP}). The limiting system
(\ref{eq:limitsystem}) with initial condition $(\rho^0,u^0)$ is
well-posed with momentum conservation condition
\[\int_\Omega\rho(t,x)u(t,x)dx=0,\quad\forall~t\geq0.\]
 We refer to \cite{fetecau2016first} for
analysis and numerical schemes for the limiting system.

The comparison of the density $\rho_\eps(x)$ and macroscopic velocity $u_\eps(x)$ at time $t = 1$ is given in Figure \ref{fig:APtest}. 
Different $\eps$'s are used for the scheme (\ref{scm:firstAP}). The results clearly demonstrate that as $\eps$ vanishes, the solution 
obtained from (\ref{scm:firstAP}) approach the solution to the limiting system, demonstrating the AP property of (\ref{scm:firstAP}).
\begin{figure}[ht!]
    \includegraphics[width=.48\textwidth]{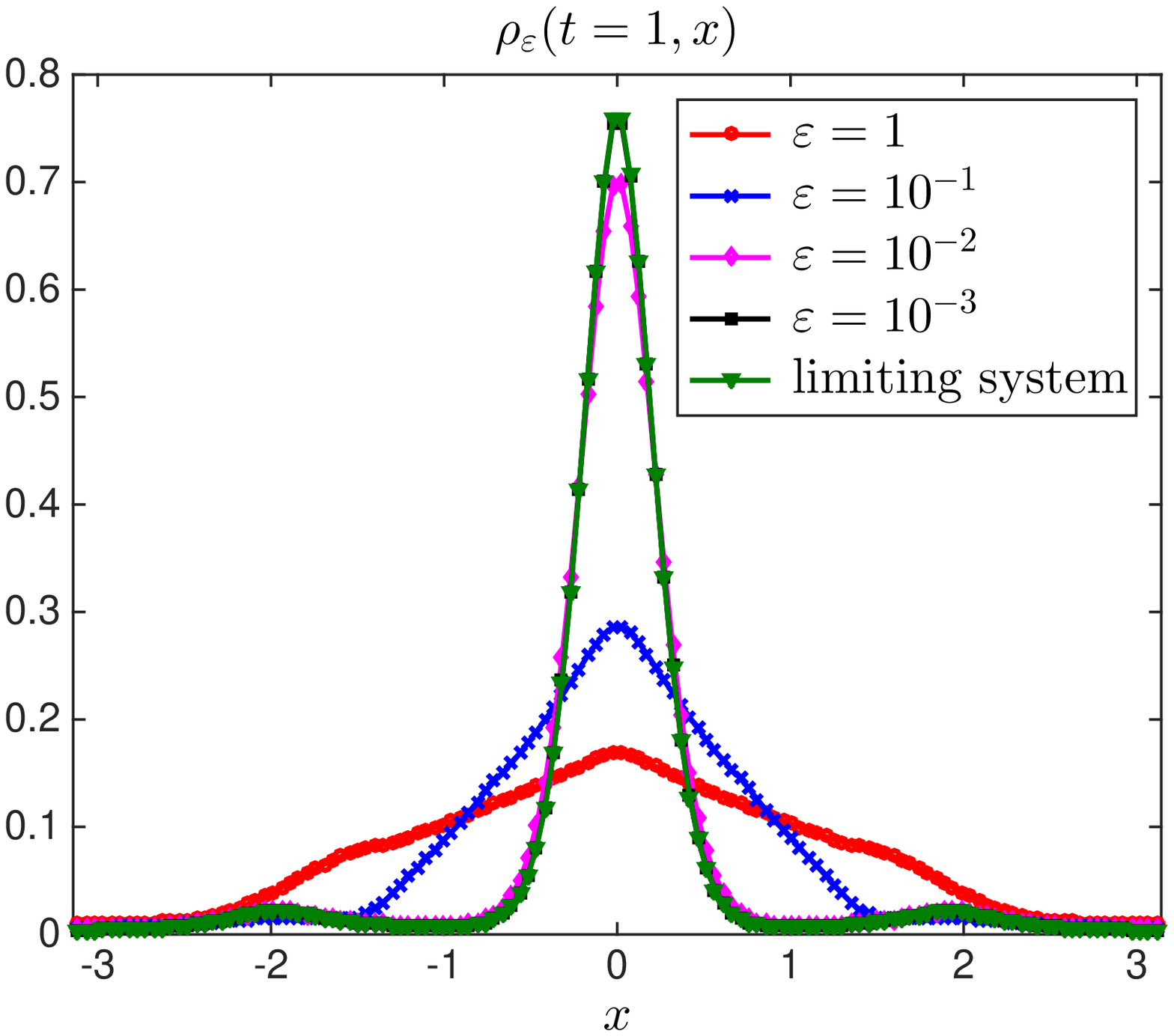}
    \includegraphics[width=.48\textwidth]{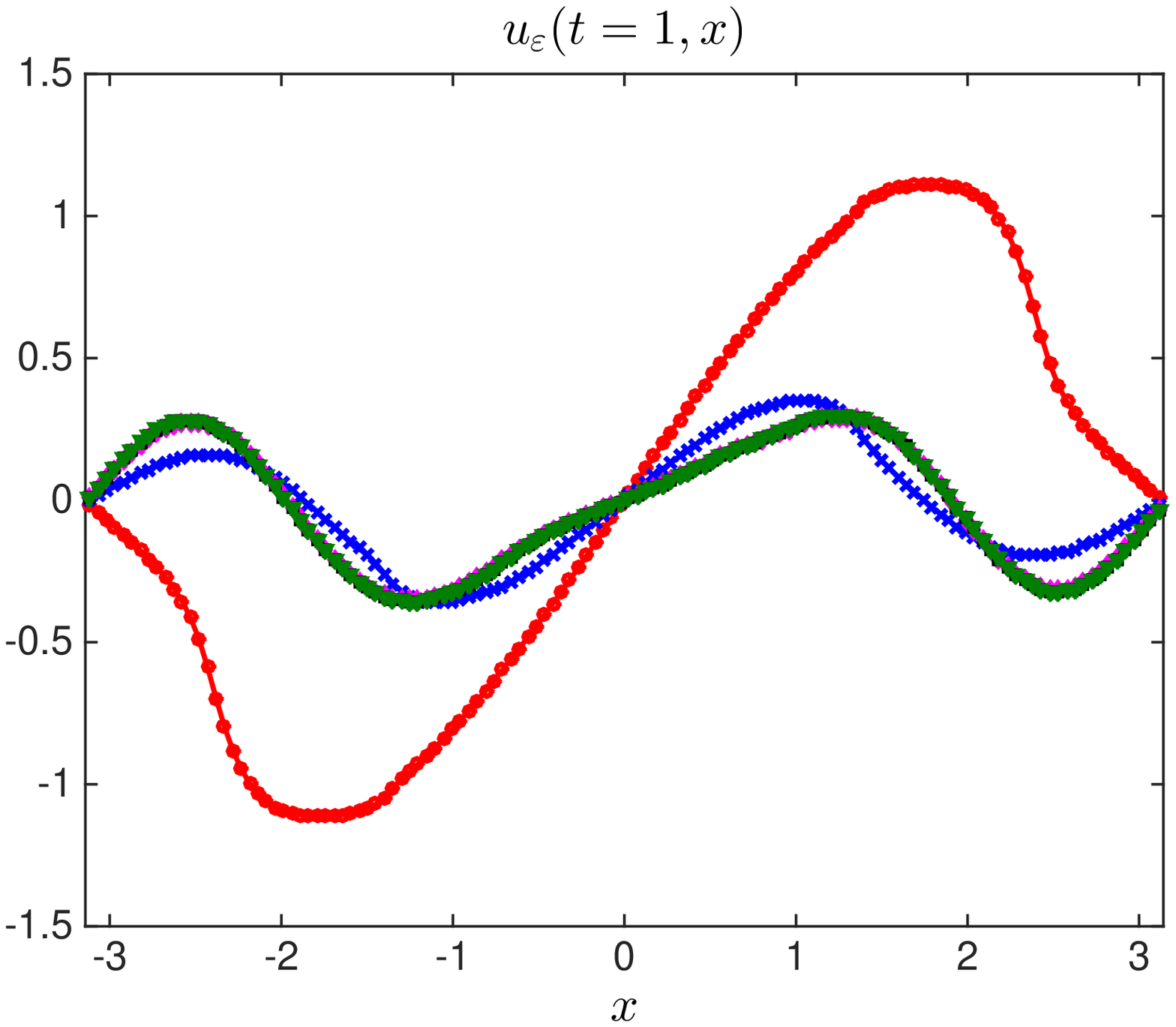}
    \caption{Example 3:  The density $\rho_\eps(x)$ (left) and the macroscopic velocity $u_\eps(x)$ (right) at time $t=1$ computed by the 
    scheme (\ref{scm:firstAP}) with different $\eps$'s are present, as well as that of the limiting system (\ref{eq:limitsystem}). The lines 
    corresponding to $\eps = 10^{-3}$ almost overlap with the lines of limiting system.}
    \label{fig:APtest}
\end{figure}

\subsection{Example 4 -- Application}
In this last example, we apply the numerical method developed in this work to an application problem. We solve the aggregation system 
(\ref{eq:agg}) with a rescaled Morse potential
\[K(x) = -e^{-|x|} + e^{-2|x|}\] 
and subject to the following initial data
\begin{equation}
\label{eq:num_initial_application}
\left\{
\begin{aligned}
g^0(x,\xi) &= \frac{\rho^0(x)}{2\sqrt{0.4\pi}} \left( e^{-\frac{(\xi +2)^2}{0.4}} + e^{-\frac{(\xi - 2)^2}{0.4}}\right)  \\
\rho^0(x) &= 10^{-8}+e^{-40x^2}, \quad u^0(x) = 0, \quad \omega^0(x) = 1,
\end{aligned}
\right.
\end{equation}
which describe two groups of agents in the same location moving to opposite directions. 

The strength of interactions between agents are characterized by the value of $\eps$. In Figure \ref{fig:application_eps0}, we take 
$\eps = 1$, hence a weak interaction is used. Time snapshots of the distribution $g_1(x,\xi)$, the density $\rho_1(x)$, the momentum 
$\rho_1(x)u_1(x)$ and the scaling factor $\omega_1(x)$ at different times are provided. It can be observed that the two groups continue 
moving toward opposite directions and eventually are separated from each other. The scaling factor $\omega_1(x)$ decays to $0$ uniformly in 
$x$. The alignment begins to dominate after a long time simulation, driving the momentum $\rho_1u_1$ to zero.

In Figure \ref{fig:application_eps4}, we plot the solution of same problem with a strong interaction by taking $\eps = 10^{-4}$. The effects 
of alignment and attraction/repulsion are much stronger than the free transport. The alignment plays a role in two aspects. First, it pushes 
$\omega_\eps$ to 0 immediately, describing all agents in the same location moving with the same velocity. This makes the two groups stick 
together. Second, after a long time, the alignment drives the momentum $\rho_\eps u_\eps$ to zero for all $x$, hence forming a flocking 
pattern. The attraction/repulsion determines the shape of this pattern. In Figure \ref{fig:application_eps4}, we also include the stationary 
solution (see e.g. \cite{fetecau2016first}) of the limiting system \eqref{ageq} and note that it agrees very well with the long time profile of the 
aggregation system.
%

\begin{figure}[ht!]
    \includegraphics[width=\textwidth]{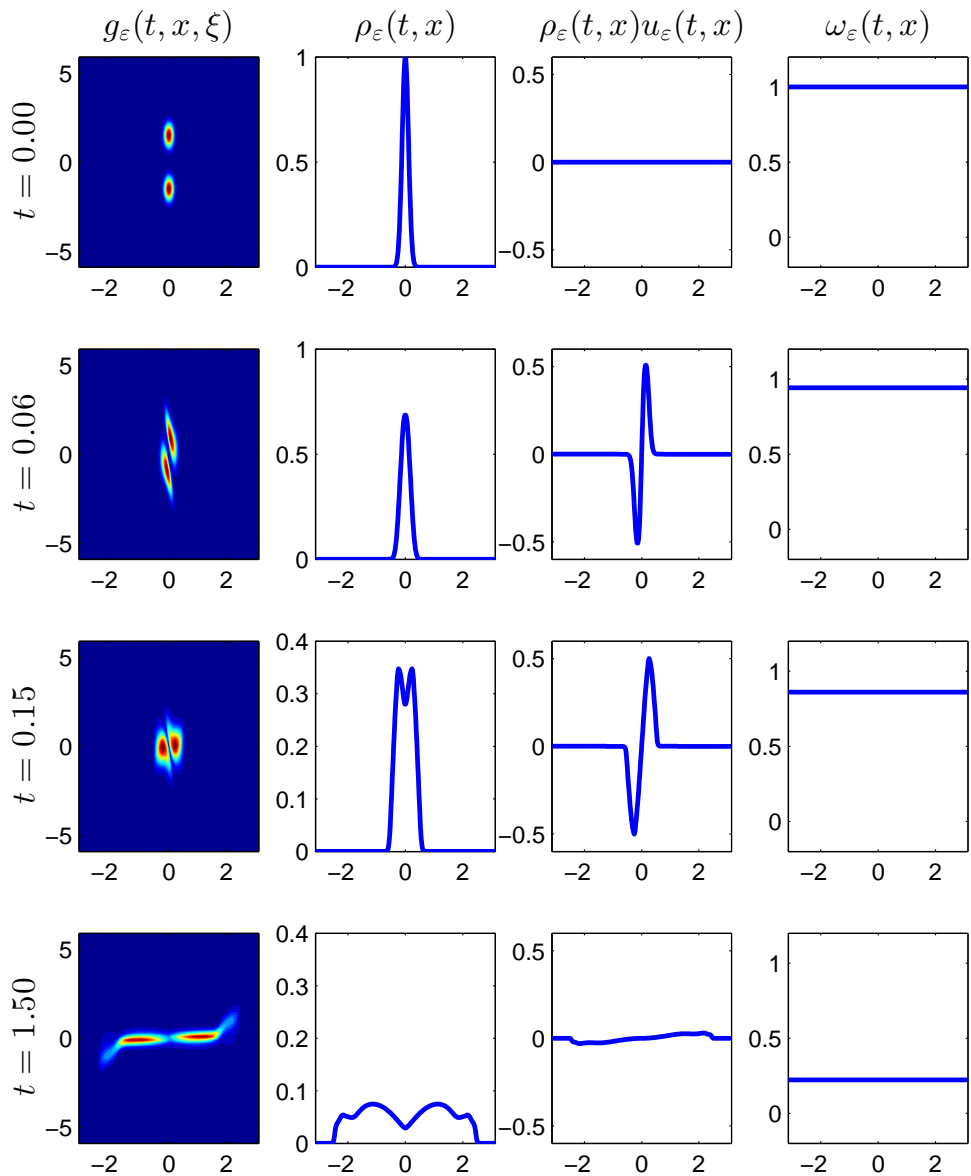}
    \caption{Example 4: Time snapshots of the solution to the aggregation system. From left to right: the distribution $g_\eps(x,\xi)$, the 
    density $\rho_\eps(x)$, the momentum $\rho_\eps(x)u_\eps(x)$ and the scaling factor $\omega_\eps(x)$. In this test $\eps = 1$.\vfill~}
    \label{fig:application_eps0}
\end{figure}
\begin{figure}[ht!]
    \includegraphics[width=\textwidth]{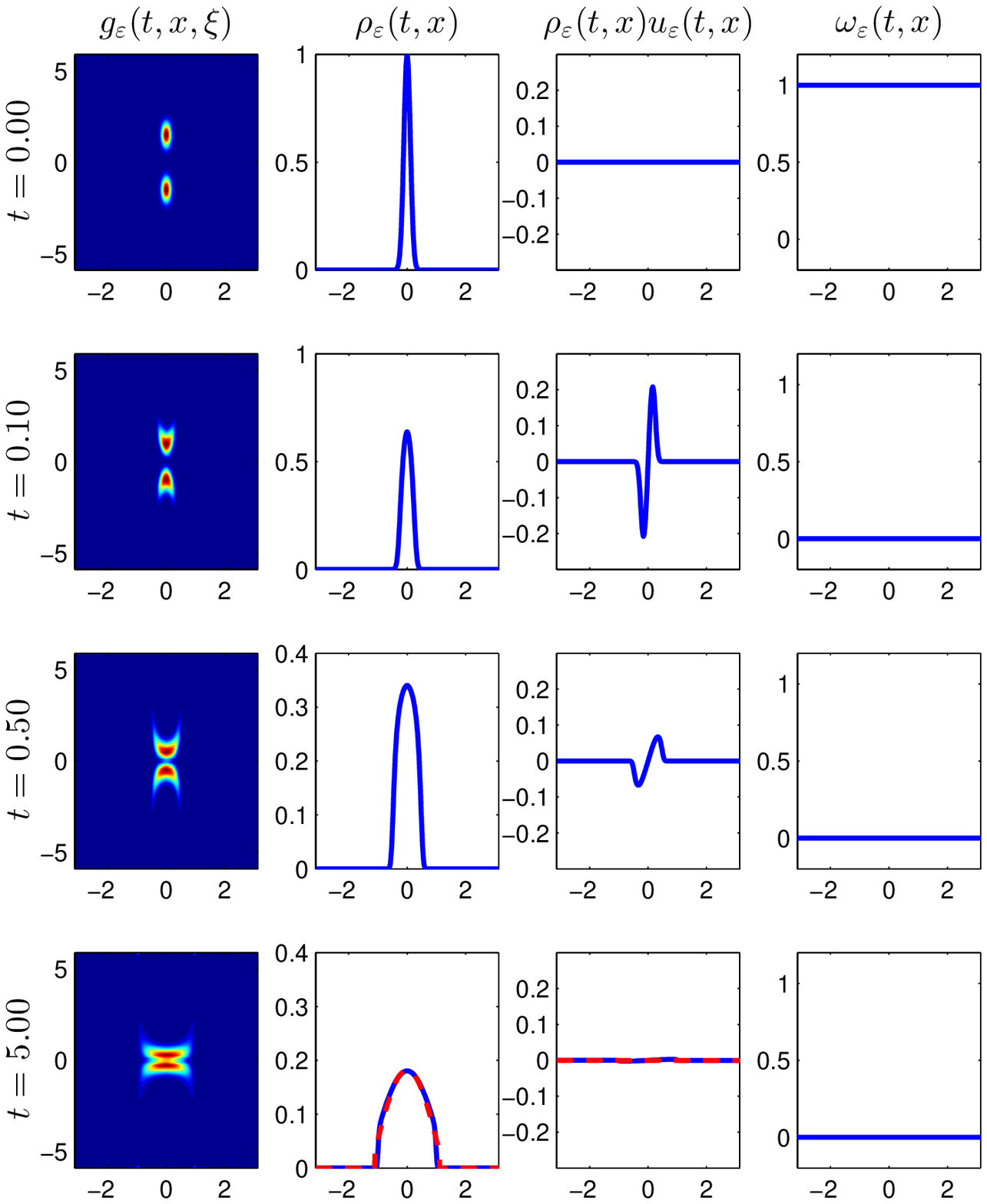}
    \caption{Example 4: Time snapshots of the solution to the aggregation system. From left to right: the distribution $g_\eps(x,\xi)$, the density $\rho_\eps(x)$, the momentum $\rho_\eps(x)u_\eps(x)$ and the scaling factor $\omega_\eps(x)$. In this test $\eps = 10^{-4}$. The stationary solution $\rho$ and $\rho u$ of the limiting system (\ref{eq:limitsystem}) is illustrated by red dashed lines in the last row.}
    \label{fig:application_eps4}
\end{figure}
\medskip
\begin{acknowledgment}
The work of A. Chertock was supported in part by NSF Grant DMS-1521051. A. Chertock and C. Tan acknowledge the support by NSF RNMS Grant DMS-1107444 (KI-Net). 
\end{acknowledgment}

\bibliographystyle{plain}
\bibliography{flocking}

\begin{thebibliography}{10}

\bibitem{bobylev2000some}
Alexander~V Bobylev, Jos{\'e}~A Carrillo, and Irene~M Gamba.
\newblock On some properties of kinetic and hydrodynamic equations for
  inelastic interactions.
\newblock {\em Journal of Statistical Physics}, 98(3-4):743--773, 2000.

\bibitem{bodnar2005derivation}
Marek Bodnar and Juna Jose~Lopez Velazquez.
\newblock Derivation of macroscopic equations for individual cell-based models:
  a formal approach.
\newblock {\em Mathematical methods in the applied sciences},
  28(15):1757--1779, 2005.

\bibitem{carrillo2016critical}
Jos{\'e}~A Carrillo, Young-Pil Choi, Eitan Tadmor, and Changhui Tan.
\newblock Critical thresholds in {1D} euler equations with non-local forces.
\newblock {\em Mathematical Models and Methods in Applied Sciences},
  26(01):185--206, 2016.

\bibitem{carrillo2010asymptotic}
Jos{\'e}~A Carrillo, Massimo Fornasier, Jes{\'u}s Rosado, and Giuseppe Toscani.
\newblock Asymptotic flocking dynamics for the kinetic cucker-smale model.
\newblock {\em SIAM Journal on Mathematical Analysis}, 42(1):218--236, 2010.

\bibitem{crouseilles2016numerical}
Nicolas Crouseilles, H{\'e}l{\`e}ne Hivert, and Mohammed Lemou.
\newblock Numerical schemes for kinetic equations in the anomalous diffusion
  limit. part i: the case of heavy-tailed equilibrium.
\newblock {\em SIAM Journal on Scientific Computing}, 38(2):A737--A764, 2016.

\bibitem{crouseilles2016numerical_II}
Nicolas Crouseilles, H{\'e}l{\`e}ne Hivert, and Mohammed Lemou.
\newblock Numerical schemes for kinetic equations in the anomalous diffusion
  limit. part ii: Degenerate collision frequency.
\newblock {\em SIAM Journal on Scientific Computing}, 38(4):A2464--A2491, 2016.

\bibitem{cucker2007emergent}
Felipe Cucker and Steve Smale.
\newblock Emergent behavior in flocks.
\newblock {\em Automatic Control, IEEE Transactions on}, 52(5):852--862, 2007.

\bibitem{do2017global}
Tam Do, Alexander Kiselev, Lenya Ryzhik, and Changhui Tan.
\newblock Global regularity for the fractional {E}uler alignment system.
\newblock {\em arXiv preprint arXiv:1701.05155}, 2017.

\bibitem{fetecau2016first}
Razvan~C Fetecau, Weiran Sun, and Changhui Tan.
\newblock First-order aggregation models with alignment.
\newblock {\em Physica D: Nonlinear Phenomena}, 325:146--163, 2016.

\bibitem{fetecau2015first}
RC~Fetecau and Weiran Sun.
\newblock First-order aggregation models and zero inertia limits.
\newblock {\em Journal of Differential Equations}, 259(11):6774--6802, 2015.

\bibitem{filbet2013rescaling}
Francis Filbet and Thomas Rey.
\newblock A rescaling velocity method for dissipative kinetic equations.
  applications to granular media.
\newblock {\em Journal of Computational Physics}, 248:177--199, 2013.

\bibitem{filbet2004rescaling}
Francis Filbet and Giovanni Russo.
\newblock A rescaling velocity method for kinetic equations: the homogeneous
  case.
\newblock {\em Proceedings Modelling and Numerics of Kinetic Dissipative
  Systems}, page~11, 2004.

\bibitem{goudon2013asymptotic}
Thierry Goudon, Shi Jin, Jian-Guo Liu, and Bokai Yan.
\newblock Asymptotic-preserving schemes for kinetic-fluid modeling of disperse
  two-phase flows.
\newblock {\em Journal of Computational Physics}, 246:145--164, 2013.

\bibitem{goudon2014asymptotic}
Thierry Goudon, Shi Jin, Jian-Guo Liu, and Bokai Yan.
\newblock Asymptotic-preserving schemes for kinetic--fluid modeling of disperse
  two-phase flows with variable fluid density.
\newblock {\em International Journal for Numerical Methods in Fluids},
  75(2):81--102, 2014.

\bibitem{ha2008particle}
Seung-Yeal Ha and Eitan Tadmor.
\newblock From particle to kinetic and hydrodynamic descriptions of flocking.
\newblock {\em Kinetic and Related Models}, 1(3):415--435, 2008.

\bibitem{jabin2000macroscopic}
Pierre-Emmanuel Jabin.
\newblock Macroscopic limit of vlasov type equations with friction.
\newblock {\em Annales de l'IHP Analyse non lin{\'e}aire}, 17(5):651--672,
  2000.

\bibitem{jin1999efficient}
Shi Jin.
\newblock Efficient asymptotic-preserving (ap) schemes for some multiscale
  kinetic equations.
\newblock {\em SIAM Journal on Scientific Computing}, 21(2):441--454, 1999.

\bibitem{jin2010asymptotic}
Shi Jin.
\newblock Asymptotic preserving ({AP}) schemes for multiscale kinetic and
  hyperbolic equations: a review.
\newblock {\em Lecture Notes for Summer School on "Methods and Models of
  Kinetic Theory" (M\&MKT), Porto Ercole (Grosseto, Italy), June 2010. Rivista
  di Matematica della Universit{\~{A}} di Parma}, 3(17):177--216, 2012.

\bibitem{mogilner1999non}
Alexander Mogilner and Leah Edelstein-Keshet.
\newblock A non-local model for a swarm.
\newblock {\em Journal of Mathematical Biology}, 38(6):534--570, 1999.

\bibitem{poyato2017euler}
David Poyato and Juan Soler.
\newblock {E}uler-type equations and commutators in singular and hyperbolic
  limits of kinetic {C}ucker--{S}male models.
\newblock {\em Mathematical Models and Methods in Applied Sciences},
  27(6):1089--1152, 2017.

\bibitem{rey2016exact}
Thomas Rey and Changhui Tan.
\newblock An exact rescaling velocity method for some kinetic flocking models.
\newblock {\em SIAM Journal on Numerical Analysis}, 54(2):641--664, 2016.

\bibitem{reynolds1987flocks}
Craig~W Reynolds.
\newblock Flocks, herds and schools: A distributed behavioral model.
\newblock {\em ACM SIGGRAPH computer graphics}, 21(4):25--34, 1987.

\bibitem{tadmor2014critical}
Eitan Tadmor and Changhui Tan.
\newblock Critical thresholds in flocking hydrodynamics with non-local
  alignment.
\newblock {\em Phil. Trans. R. Soc. A}, 372(2028):20130401, 2014.

\bibitem{tan2017discontinuous}
Changhui Tan.
\newblock A discontinuous {G}alerkin method on kinetic flocking models.
\newblock {\em Mathematical Models and Methods in Applied Sciences},
  27(7):1199--1221, 2017.

\bibitem{topaz2006nonlocal}
Chad~M Topaz, Andrea~L Bertozzi, and Mark~A Lewis.
\newblock A nonlocal continuum model for biological aggregation.
\newblock {\em Bulletin of mathematical biology}, 68(7):1601--1623, 2006.

\bibitem{wang2016asymptotic}
Li~Wang and Bokai Yan.
\newblock An asymptotic-preserving scheme for linear kinetic equation with
  fractional diffusion limit.
\newblock {\em Journal of Computational Physics}, 312:157--174, 2016.

\bibitem{wang2017asymptotic}
Li~Wang and Bokai Yan.
\newblock An asymptotic-preserving scheme for kinetic equations with
  anisotropic scattering: fat tail equilibrium and degenerate collision
  frequency.
\newblock {\em preprint}, 2017.

\end{thebibliography}

\end{document}